\documentclass[10pt,a4paper]{article}

\usepackage{latexsym,amsfonts,amsmath,amssymb,mathrsfs,url,color,amsthm}
\usepackage{graphicx,cases}
\usepackage[normalem]{ulem}

\newtheorem{theorem}{Theorem}
\newtheorem{lemma}[theorem]{Lemma}
\newtheorem{algorithm}[theorem]{Algorithm}
\newtheorem{remark}[theorem]{Remark}
\newtheorem{proposition}[theorem]{Proposition}
\newtheorem{definition}[theorem]{Definition}

\newcommand{\rd}{\,\mathrm{d}}

\newcommand{\bsa}{\boldsymbol{a}}
\newcommand{\bsb}{\boldsymbol{b}}

\newcommand{\bsk}{\boldsymbol{k}}
\newcommand{\bsl}{\boldsymbol{l}}
\newcommand{\bsq}{\boldsymbol{q}}

\newcommand{\bsw}{\boldsymbol{w}}
\newcommand{\bsx}{\boldsymbol{x}}
\newcommand{\bsy}{\boldsymbol{y}}
\newcommand{\bsz}{\boldsymbol{z}}

\newcommand{\bsP}{\boldsymbol{P}}
\newcommand{\bsQ}{\boldsymbol{Q}}
\newcommand{\bsR}{\boldsymbol{R}}

\newcommand{\bsU}{\boldsymbol{U}}
\newcommand{\bsV}{\boldsymbol{V}}
\newcommand{\bsalpha}{\boldsymbol{\alpha}}

\newcommand{\bsnu}{\boldsymbol{\nu}}

\newcommand{\bstau}{\boldsymbol{\tau}}

\newcommand{\bszero}{\boldsymbol{0}}

\newcommand{\nat}{\mathbb{N}}

\newcommand{\RR}{\mathbb{R}}
\newcommand{\ZZ}{\mathbb{Z}}
\newcommand{\Bcal}{\mathcal{B}}

\newcommand{\Ecal}{\mathcal{E}}

\newcommand{\Hcal}{\mathcal{H}}

\newcommand{\Kcal}{\mathcal{K}}

\newcommand{\Pcal}{\mathcal{P}}

\newcommand{\tr}{\mathrm{tr}}
\newcommand{\wal}{\mathrm{wal}}

\newcommand{\Zbinfty}{G}

\allowdisplaybreaks

\begin{document}

\title{Digital nets with infinite digit expansions and construction of folded digital nets for quasi-Monte Carlo integration\thanks{The works of the second and third authors were supported by the Program for Leading Graduate Schools, MEXT, Japan.}}

\author{Takashi Goda\thanks{Graduate School of Engineering, The University of Tokyo, 7-3-1 Hongo, Bunkyo-ku, Tokyo 113-8656 ({\tt goda@frcer.t.u-tokyo.ac.jp}).}, Kosuke Suzuki\thanks{Graduate School of Mathematical Sciences, The University of Tokyo, 3-8-1 Komaba, Meguro-ku, Tokyo 153-8914 ({\tt ksuzuki@ms.u-tokyo.ac.jp}), JSPS research fellow.}, Takehito Yoshiki\thanks{Graduate School of Mathematical Sciences, The University of Tokyo, 3-8-1 Komaba, Meguro-ku, Tokyo 153-8914 ({\tt yosiki@ms.u-tokyo.ac.jp}).}}

\date{\today}

\maketitle

\begin{abstract}
In this paper we study quasi-Monte Carlo integration of smooth functions using digital nets. We fold digital nets over $\ZZ_b$ by means of the $b$-adic tent transformation, which has recently been introduced by the authors, and employ such \emph{folded digital nets} as quadrature points. We first analyze the worst-case error of quasi-Monte Carlo rules using folded digital nets in reproducing kernel Hilbert spaces. Here we need to permit digital nets with ``infinite digit expansions'', which are beyond the scope of the classical definition of digital nets. We overcome this issue by considering the infinite product of cyclic groups and the characters on it. We then give an explicit means of constructing good folded digital nets as follows: we use higher order polynomial lattice point sets for digital nets and show that the component-by-component construction can find good \emph{folded higher order polynomial lattice rules} that achieve the optimal convergence rate of the worst-case error in certain Sobolev spaces of smoothness of arbitrarily high order.
%The significant advantage of our rules is that we can reduce the required degree of modulus by half as compared to ordinary higher order polynomial lattice rules. Thus it is also possible to reduce a construction cost of the algorithm, requiring $O(s \alpha N^{\alpha/2}\log N)$ arithmetic operations using $O(N^{\alpha/2})$ memory, where $N$ and $s$ denote the number of points and the dimension, respectively.
\end{abstract}

%%%%%%%%%%%%%%%%%%%%%%%%%%%%%%%%%%%%%%%%%%%%%%%%%%%%%%%%%%%
%%%%%%%%%%%%%%%%%%%%%%%%%%%%%%%%%%%%%%%%%%%%%%%%%%%%%%%%%%%
%%%%%%%%%%%%%%%%%%%%%%%%%%%%%%%%%%%%%%%%%%%%%%%%%%%%%%%%%%%
\section{Introduction}\label{sec:intro}
Quasi-Monte Carlo (QMC) integration of a real-valued function $f$ defined over the $s$-dimensional unit cube is given by
  \begin{align*}
     Q(f;P) := \frac{1}{|P|}\sum_{\bsx\in P}f(\bsx) ,
  \end{align*}
where $P\subset [0,1]^s$ is a point set and $|P|$ denotes the cardinality of $P$, to approximate the integral 
  \begin{align*}
     I(f) := \int_{[0,1]^s}f(\bsx)\rd \bsx ,
  \end{align*}
as accurately as possible. As one of the main families of QMC point sets, digital nets and sequences have been extensively studied in the literature, see for instance \cite{DP10,N92b}. We shall discuss the definition of digital nets in Subsection \ref{subsec:digital_net}. There have been many good explicit constructions of digital nets and sequences, including those proposed by Sobol', Faure, Niederreiter, Niederreiter and Xing as well as others, see \cite[Section~8]{DP10} for more information. These point sets generally hold good properties of uniform distribution modulo one. The typical convergence rate of the QMC integration error $|I(f)-Q(f;P)|$ using these point sets is $O(|P|^{-1+\varepsilon})$ with arbitrarily small $\varepsilon >0$.

Our goal of this paper is to give an explicit means of constructing good deterministic point sets for QMC integration of smooth functions in certain Sobolev spaces $\Hcal_{\alpha}$ of smoothness of arbitrarily high order $\alpha\ge 2$. For such smooth functions, it is possible to achieve higher order convergence of $O(|P|^{-\alpha+\varepsilon})$ (with arbitrarily small $\varepsilon >0$) of the QMC integration error by using \emph{higher order} digital nets \cite{BD09,D08}. The explicit construction of higher order digital nets introduced in \cite{D08} uses digital nets whose number of components is a multiple of the dimension, and interlaces them digitally in a certain way. Another construction of higher order digital nets, known under the name of \emph{higher order polynomial lattice point sets} (HOPLPSs), is introduced in \cite{DP07} by generalizing the definition of polynomial lattice point sets, which was originally given in \cite{N92a}. Recently, one of the authors has utilized original polynomial lattice point sets as interlaced components in the former construction principle, and has proved that it is possible to obtain good \emph{interlaced polynomial lattice point sets} (IPLPSs) for higher order digital nets \cite{Gxx1}, see also \cite{GDxx}.

The most important advantage of IPLPSs over HOPLPSs lies in the construction cost. Fast component-by-component (CBC) construction requires $O(s\alpha |P|\log |P|)$ arithmetic operations using $O(|P|)$ memory for IPLPSs \cite{Gxx1}, whereas requiring $O(s\alpha |P|^{\alpha}\log |P|)$ arithmetic operations using $O(|P|^{\alpha})$ memory for HOPLPSs \cite{BDLNP12}. In order to reduce the construction cost for HOPLPSs, one of the authors considered applying a random digital shift and then folding the resulting point sets by using the tent transformation in \cite{Gxx2}. (We note that the tent transformation was originally used for lattice rules in \cite{H02}.) The obtained cost for the fast CBC construction becomes $O(s\alpha |P|^{\alpha/2}\log |P|)$ arithmetic operations using $O(|P|^{\alpha/2})$ memory. This is a generalization of the study in \cite{CDLP07}. However, this result not only restricts the base $b$ to $2$, but also needs a randomization by a random digital shift. Thus, we cannot construct good \emph{deterministic} point sets in this way in contrast to \cite{BDLNP12,Gxx1}.

Regarding the restriction of the base, the authors have recently introduced the \emph{$b$-adic tent transformation} ($b$-TT) in \cite{GSYxx} for any positive integer $b\ge 2$, by generalizing the original (dyadic) tent transformation, and studied the mean square worst-case error of \emph{digitally shifted and then folded} digital nets in reproducing kernel Hilbert spaces. As a continuation of the study in \cite{GSYxx}, we resolve the above concerns on \cite{Gxx2} in this paper.

We first consider QMC point sets which are obtained by folding digital nets by means of the $b$-TT, and study the worst-case error of QMC rules using such \emph{folded digital nets} in reproducing kernel Hilbert spaces. In our analysis, we need to permit digital nets with ``infinite digit expansions'', which are beyond the scope of the classical definition of digital nets, see for example \cite[Chapter~4]{DP10} and \cite[Chapter~4]{N92b}. To overcome this issue, we consider infinite products of cyclic groups $\Zbinfty$ and the characters on $\Zbinfty$, and define digital nets in $\Zbinfty^s$ by using infinite-column generating matrices, i.e., generating matrices whose each column can contain infinitely many entries different from zero, see Definition \ref{def:digital_net}. Then we discuss the dual net of folded digital nets and the worst-case error in reproducing kernel Hilbert spaces. This is the first contribution of this paper.

Using the results of the above argument, we do the following next. We employ HOPLPSs in prime base $b$ that are folded using the $b$-TT\@. We call such deterministic point sets \emph{folded higher order polynomial lattice point sets} (FHOPLPSs). We consider the Sobolev space $\Hcal_{\alpha}$ as a function space, and prove that the component-by-component construction can find good FHOPLPSs for which QMC integration achieves the optimal convergence rate of the worst-case error in $\Hcal_{\alpha}$. Moreover, we show how to obtain the fast component-by-component construction using the fast Fourier transform in a way analogous to \cite{BDLNP12,Gxx2}. We obtain a construction cost of the algorithm of $O(s\alpha |P|^{\alpha/2}\log |P|)$ arithmetic operations using $O(|P|^{\alpha/2})$ memory. This is the second contribution of this paper. We note that QMC point sets constructed in this way are considered useful in the context of uncertainty quantification, in particular partial differential equations with random coefficients \cite{DKLNS14}, although the investigation of this application is beyond the scope of this paper.

The remainder of this paper is organized as follows. In Section \ref{sec:pre}, we recall some necessary background and notation, including infinite products of cyclic groups $G$, Walsh functions, digital nets, and HOPLPSs. As mentioned above, we define digital nets in $\Zbinfty^s$, instead of those in $[0,1]^s$, by using infinite-column generating matrices. In Section \ref{sec:wc_error}, we first describe the $b$-TT and some properties of folded digital nets, that is, digital nets that are folded by using the $b$-TT, and then study the worst-case error of QMC rules using folded digital nets in reproducing kernel Hilbert spaces. In Section \ref{sec:bound}, we introduce the Sobolev spaces $\Hcal_{\alpha}$ and give an upper bound on the worst-case error for QMC rules using folded digital nets in $\Hcal_{\alpha}$. In Section \ref{sec:cbc}, we prove that the CBC construction can find good FHOPLPSs for which QMC integration achieves the optimal convergence rate of the worst-case error in $\Hcal_{\alpha}$. Finally in Section \ref{sec:fast}, we show how to obtain the fast CBC construction using the fast Fourier transform in a way analogous to \cite{BDLNP12,Gxx2}.

%%%%%%%%%%%%%%%%%%%%%%%%%%%%%%%%%%%%%%%%%%%%%%%%%%%%%%%%%%%
%%%%%%%%%%%%%%%%%%%%%%%%%%%%%%%%%%%%%%%%%%%%%%%%%%%%%%%%%%%
%%%%%%%%%%%%%%%%%%%%%%%%%%%%%%%%%%%%%%%%%%%%%%%%%%%%%%%%%%%
\section{Preliminaries}\label{sec:pre}
Throughout this paper, we use the following notation. Let $\nat$ be the set of positive integers and let $\nat_0:=\nat\cup \{0\}$. Let $\mathbb{C}$ be the set of all complex numbers. For a positive integer $b\ge 2$, let $\ZZ_b$ be a cyclic group with $b$ elements, which is identified with the set $\{0,1,\dots,b-1\}$ equipped with addition modulo $b$.

%The operators $\oplus$ and $\ominus$ denote digitwise addition and subtraction modulo $b$, respectively. That is, for $x, x'\in [0,1]$ whose $b$-adic expansions are $x=\sum_{i=1}^{\infty}\xi_i b^{-i}$ and $x'=\sum_{i=1}^{\infty}\xi'_i b^{-i}$ with $\xi_i,\xi'_i\in \ZZ_b$ for $i\in \nat$, $\oplus$ and $\ominus$ are defined as
%  \begin{align*}
%    x\oplus x' = \sum_{i=1}^{\infty}\frac{\eta_i}{b^i}\quad \text{and}\quad x\ominus x' = \sum_{i=1}^{\infty}\frac{\eta'_i}{b^i},
%  \end{align*}
%where $\eta_i= \xi_i+\xi'_i \pmod b$ and $\eta'_i= \xi_i-\xi'_i \pmod b$, respectively. Similarly, we define digitwise addition and subtraction for non-negative integers based on their $b$-adic expansions. In case of vectors in $[0,1]^s$ or $\nat_0^s$, the operators $\oplus$ and $\ominus$ are applied componentwise.

%%%%%%%%%%%%%%%%%%%%%%%%%%%%%%%%%%%%%%%%%%%%%%%%%%%%%%%%%%%
\subsection{Infinite products of cyclic groups}\label{subsec: infinite digit}
This subsection is based on \cite{Walshbook},
which treats only the one-dimensional dyadic case,
and \cite{Pontryagin},
which constructs Pontryagin duality theory for locally compact abelian groups.
First we consider the one-dimensional case.
Let us define $\Zbinfty := \prod_{i=1}^{\infty} \ZZ_b$.
$\Zbinfty$ is a compact abelian group with the product topology,
where $\ZZ_b$ is considered to be a discrete group.
We denote by $\oplus$ and $\ominus$ addition and subtraction in $\Zbinfty$, respectively.
Let $\nu$ be the product measure on $\Zbinfty$ 
inherited from the uniform measure on $\ZZ_b$; For every cylinder set
$E = \prod_{i=1}^n Z_i \times \prod_{i=n+1}^\infty  \ZZ_b$ with
$Z_i \subset \ZZ_b$ ($1 \leq i \leq n$),
it holds that $\nu(E) = \prod_{i=1}^n (|Z_i|/b)$.

A character on $\Zbinfty$
is a continuous group homomorphism from $\Zbinfty$ to $\{z\in \mathbb{C}: |z|=1\}$, which is a multiplicative group of complex numbers whose absolute value is 1.
We define the $k$-th character $W_{k}$ as follows.

\begin{definition}
Let $b\ge 2$ be a positive integer, and let $\omega:=\exp(2\pi \sqrt{-1}/b)$ be the primitive $b$-th root of unity.
Let $z = (\zeta_1, \zeta_2, \dots)^{\top} \in \Zbinfty$ and $k \in \nat_0$ 
whose $b$-adic expansion is $k = \kappa_0+\kappa_1b+\dots+\kappa_{a-1}b^{a-1}$ with $\kappa_0,\ldots,\kappa_{a-1}\in \ZZ_b$. 
Then the $k$-th character
$W_k \colon \Zbinfty \to \{1, \omega, \dots, \omega^{b-1}\}$
is given by
$$
W_k(z) := \omega^{\kappa_0 \zeta_1 + \cdots + \kappa_{a-1} \zeta_a}.
$$
\end{definition}
\noindent
We note that every character on $\Zbinfty$ is equal to some $W_{k}$, see \cite{Pontryagin}.

The group $\Zbinfty$ can be related to the interval $[0,1]$ through two maps $\pi \colon \Zbinfty \to [0,1]$ and $\sigma \colon [0,1] \to \Zbinfty$.
Let $z = (\zeta_1, \zeta_2, \dots)^{\top} \in \Zbinfty$ and
$x \in [0,1]$ whose $b$-adic expansions are $x=\sum_{i=1}^{\infty}\xi_i b^{-i}$
with $\xi_i \in \ZZ_b$,
which is unique in the sense that infinitely many of the $\xi_i$ are different from $b-1$ if $x \neq 1$ and that all $\xi_i$ are equal to $b-1$ if $x = 1$.
Then the projection map $\pi \colon \Zbinfty \to [0,1]$
is defined as $\pi(z) = \sum_{i=1}^\infty \zeta_i b^{-i}$
and the section map $\sigma \colon [0,1] \to \Zbinfty$
is defined as $\sigma(x) = (\xi_1, \xi_2, \dots)^{\top}$.
In the dyadic case, $\sigma$ is called Fine's map.
By definition, we have that $\pi$ is surjective and $\sigma$ is injective.

These definitions can be generalized to the higher-dimensional case.
Let $\Zbinfty^s$ denote the $s$-ary Cartesian product of $\Zbinfty$.
$\Zbinfty^s$ is also a compact abelian group with the product topology.
The operators $\oplus$ and $\ominus$ denote
addition and subtraction in $\Zbinfty^s$, respectively.
We denote by $\bsnu$ the product measure on $\Zbinfty^s$
inherited from $\nu$.
For integrals on $\Zbinfty^s$, we only consider the measure $\bsnu$.
Hence, in order to emphasize the variable,
we write $\int_{G^s} f(\bsz) \rd \bsz$ instead of $\int_{G^s} f \rd \bsnu$.
We define the $\bsk$-th character $W_{\bsk}$ as follows.
\begin{definition}
Let $b\ge 2$ be a positive integer.
For a dimension $s\in \nat$,
let $\bsz=(z_1,\ldots, z_s) \in \Zbinfty^s$ and $\bsk=(k_1,\ldots, k_s)\in \nat_0^s$.
Then the $\bsk$-th character
$W_{\bsk} \colon \Zbinfty^s \to \{1,\omega_b,\ldots, \omega_b^{b-1}\}$
is defined as
  \begin{align*}
    W_{\bsk}(\bsz) := \prod_{j=1}^s W_{k_j}(z_j) .
  \end{align*}
\end{definition}
\noindent
Note that every character on $\Zbinfty^s$ is equal to some $W_{\bsk}$
as with the one-dimensional case.
For the $s$-dimensional projection and section map,
we use the same symbol $\pi$ and $\sigma$ as in the one-dimensional case.
That is, for $\bsz=(z_1,\ldots, z_s) \in \Zbinfty^s$ and $\bsx= (x_1, \dots, x_s) \in [0,1]^s$,
we define the projection map $\pi \colon \Zbinfty^s \to [0,1]^s$
as $\pi(z) = (\pi(z_1), \dots, \pi(z_s))$
and the section map $\sigma \colon [0,1]^s \to \Zbinfty^s$
as $\sigma(x) = (\sigma(x_1), \dots, \sigma(x_s))$.

The orthogonality of the characters on $\Zbinfty^s$
are described below, see \cite{Pontryagin} for the proof.

\begin{proposition}\label{prop: infinite digit}
The following holds true:
\begin{enumerate}
\item For $k \in \nat_0$, we have
  \begin{align*}
    \int_{\Zbinfty} W_{k}(z)\rd z = \begin{cases}
     1 & \text{if $k=0$},  \\
     0 & \text{otherwise} .
    \end{cases}
  \end{align*}
\item For all $\bsk,\bsl\in \nat_0^s$, we have
  \begin{align*}
    \int_{\Zbinfty^s} W_{\bsk}(\bsz)\overline{W_{\bsl}(\bsz)}\rd \bsz = \begin{cases}
     1 & \text{if $\bsk=\bsl$},  \\
     0 & \text{otherwise} .
    \end{cases}
  \end{align*}
\end{enumerate}
\end{proposition}

Moreover, we need a lemma on the partial sum of characters referred to as Paley's lemma. 
Since the proof is essentially the same as \cite[Paley's lemma]{Walshbook}, we omit it.
\begin{lemma}\label{lem:Paley}
Let $n$ be a positive integer.
For $\bsz = (z_1, \dots, z_s) \in \Zbinfty^s$ with
$z_i = (\zeta_{i,1}, \zeta_{i,2}, \dots)^{\top} \in \Zbinfty$, we have
$$
\sum_{\bsk <b^n} W_{\bsk}(\bsz)
= \begin{cases}
b^{sn} &\text{if $\zeta_{i,1} = \cdots = \zeta_{i,n} = 0$ for all $1 \leq i \leq s$}, \\
0 &\text{otherwise},
\end{cases}
$$
where $\bsk=(k_1, \dots, k_s) < b^n$ means that $k_i < b^n$ holds
for every $1 \leq i \leq s$.
\end{lemma}

Properties of $\pi$ and $\sigma$ are described below.
\begin{proposition}\label{prop: map pi and rho}
We have the following:
\begin{enumerate}
\item $\pi$ is a continuous map.
\item $\pi \circ \sigma = \mathrm{id}_{[0,1]^s}$.
\item For $f \in L^1(\Zbinfty^s)$, we have
$$
\int_{\Zbinfty^s} f(\bsz) \rd \bsz = \int_{[0,1]^s} f (\sigma(\bsx)) \rd \bsx.
$$
\item For $f \in L^1([0,1]^s)$, we have
$$
\int_{[0,1]^s} f(\bsx)  \rd \bsx = \int_{\Zbinfty^s} f(\pi(\bsz)) \rd \bsz.
$$
\end{enumerate}
\end{proposition}
Items~1 and 2 of Proposition~\ref{prop: map pi and rho}
follow by definition.
For the proof of Items~3 and 4 of Proposition~\ref{prop: map pi and rho},
we refer to \cite[Theorem~5]{Walshbook}.

%%%%%%%%%%%%%%%%%%%%%%%%%%%%%%%%%%%%%%%%%%%%%%%%%%%%%%%%%%%
\subsection{Walsh functions}\label{subsec:walsh}
Walsh functions play a central role in the analysis of digital nets. We refer to \cite[Appendix~A]{DP10} for general information on Walsh functions. We first give the definition for the one-dimensional case.

\begin{definition}
Let $b\ge 2$ be a positive integer and let $\omega_b=\exp(2\pi \sqrt{-1}/b)$. We denote the $b$-adic expansion of $k\in \nat_0$ by $k = \kappa_0+\kappa_1b+\dots+\kappa_{a-1}b^{a-1}$ with $\kappa_0,\ldots,\kappa_{a-1}\in \ZZ_b$. Then the $k$-th $b$-adic Walsh function ${}_b\wal_k \colon [0,1]\to \{1,\omega_b,\dots,\omega_b^{b-1}\}$ is defined as
  \begin{align*}
    {}_b\wal_k(x) := \omega_b^{\kappa_0\xi_1+\dots+\kappa_{a-1}\xi_a} ,
  \end{align*}
for $x\in [0,1]$ whose $b$-adic expansion is given by $x=\xi_1b^{-1}+\xi_2b^{-2}+\cdots$, which is unique in the sense that infinitely many of the $\xi_i$ are different from $b-1$ if $x \neq 1$ and that all $\xi_i$ are equal to $b-1$ if $x = 1$.
\end{definition}
\noindent
We note that Walsh functions are usually defined on $[0,1)$, whereas they are defined on $[0,1]$ specially for this study.
This definition can be generalized to the higher-dimensional case.

\begin{definition}
Let $b\ge 2$ be a positive integer. For a dimension $s\in \nat$, let $\bsx=(x_1,\ldots, x_s)\in [0,1]^s$ and $\bsk=(k_1,\ldots, k_s)\in \nat_0^s$. Then the $\bsk$-th $b$-adic Walsh function ${}_b\wal_{\bsk} \colon [0,1]^s \to \{1,\omega_b,\ldots, \omega_b^{b-1}\}$ is defined as
  \begin{align*}
    {}_b\wal_{\bsk}(\bsx) := \prod_{j=1}^s {}_b\wal_{k_j}(x_j) .
  \end{align*}
\end{definition}
Since we shall always use Walsh functions in a fixed base $b$, we omit the subscript and simply write $\wal_k$ or $\wal_{\bsk}$ in this paper.
By the definition of characters and Walsh functions, we can see that
\begin{equation}\label{eq:WandWalsh}
\wal_{\bsk}(\bsx) = W_{\bsk}(\sigma (\bsx)).
\end{equation}
 Some important properties of Walsh functions, used in this paper, are described below, see \cite[Appendix~A.2]{DP10} for the proof.

\begin{proposition}\label{prop:walsh}
We have the following:
\begin{enumerate}
\item For $k\in \nat_0$, we have
  \begin{align*}
    \int_0^1 \wal_k(x)\rd x = \begin{cases}
     1 & \text{if $k=0$},  \\
     0 & \text{otherwise} .
    \end{cases}
  \end{align*}
\item For all $\bsk,\bsl\in \nat_0^s$, we have
  \begin{align*}
    \int_{[0,1]^s} \wal_{\bsk}(\bsx)\overline{\wal_{\bsl}(\bsx)}\rd \bsx = \begin{cases}
     1 & \text{if $\bsk=\bsl$},  \\
     0 & \text{otherwise} .
    \end{cases}
  \end{align*}
\item The system $\{\wal_{\bsk}: \bsk\in \nat_0^s\}$ is a complete orthonormal system in $L^2([0,1]^s)$ for any $s\in \nat$.
\end{enumerate}
\end{proposition}

From Item 3 of Proposition \ref{prop:walsh}, we define the Walsh series of $f\in L^2([0,1]^s)$ by
  \begin{align*}
     \sum_{\bsk\in \nat_0^s}\hat{f}(\bsk)\wal_{\bsk} ,
  \end{align*}
where the $\bsk$-th Walsh coefficient is given by
  \begin{align*}
     \hat{f}(\bsk) = \int_{[0,1]^s} f(\bsx)\overline{\wal_{\bsk}(\bsx)}\rd \bsx .
  \end{align*}
We refer to \cite[Appendix~A.3]{DP10} and \cite[Lemma~17]{GSYxx} for a discussion on the pointwise absolute convergence of the Walsh series.% See also Proposition \ref{prop:pointwise_convergence} for a sufficient condition on the pointwise absolute convergence of the Walsh series, which is required in the analysis of this paper.

%%%%%%%%%%%%%%%%%%%%%%%%%%%%%%%%%%%%%%%%%%%%%%%%%%%%%%%%%%%
\subsection{Digital nets}\label{subsec:digital_net}
We introduce the general definition of digital nets in $\Zbinfty^s$ by using infinite-column generating matrices for an arbitrary positive integer $b\ge 2$.

\begin{definition}\label{def:digital_net}
For $m\in \nat$, let $C_1,\ldots,C_s\in \ZZ_b^{\nat \times m}$. For each non-negative integer $h$ with $0\le h<b^m$, we denote its $b$-adic expansion by $h=\sum_{i=0}^{m-1}\eta_ib^i$. Let $\bsz_h=(z_{h,1},\ldots,z_{h,s})\in \Zbinfty^s$ be given by
  \begin{align*}
     z_{h,j} = C_j\cdot (\eta_0, \eta_1,\ldots, \eta_{m-1})^{\top} ,
  \end{align*}
for $1\le j\le s$. We call $\mathcal{P}=\{\bsz_0,\bsz_1,\ldots,\bsz_{b^m-1}\}$ a \emph{digital net in $\Zbinfty^s$} with generating matrices $C_1,\ldots,C_s$.
\end{definition}

We note that every digital net in $\Zbinfty^s$ is a $\ZZ_b$-module of $\Zbinfty^s$ as well as a subgroup of $\Zbinfty^s$ by Definition~\ref{def:digital_net}. We call $P \subset [0,1]^s$ a \emph{digital net in $[0,1]^s$} over $\ZZ_b$
if there exists a $\mathcal{P}$ which is a digital net in $\Zbinfty^s$ 
with $P = \pi(\mathcal{P})$.
In this paper, we use digital nets in $[0,1]^s$ over $\ZZ_b$ as quadrature points.

\begin{remark}\label{rem:infinite_digits}
Historically, (classical) digital nets are subsets of $[0,1]^s$ constructed by finite-column generating matrices, i.e., generating matrices whose each column consists of only finitely many entries different from zero. If $P$ is a classical digital net, we have that $\sigma(P)$ is a digital net in $\Zbinfty^s$ and that $\pi(\sigma(P)) = P$.
Hence a classical digital net $P$ can be considered as a digital net in $[0,1]^s$ over $\ZZ_b$.
Our definition of digital nets permits digital nets with ``infinite digit expansions''.
Although the concept of infinite digit expansions has already appeared in the definition of digital sequences \cite[Chapter~8]{NX01}, it has not been discussed for digital nets in the literature as far as the authors know.
\end{remark}

The \emph{dual net} of a digital net plays an important role in the subsequent analysis. For a digital net $\mathcal{P}$, its dual net, denoted by $\mathcal{P}^{\perp}$, is defined as follows.
\begin{definition}\label{def:dual_net}
Let $\mathcal{P}$ be a digital net in $\Zbinfty^s$ with generating matrices $C_1,\ldots,C_s\in \ZZ_b^{\nat \times m}$.
The dual net of $\mathcal{P}$ is defined as
  \begin{align*}
     \mathcal{P}^{\perp}:=\{\bsk\in \nat_0^s \colon C_1^{\top}\vec{k}_1\oplus \dots \oplus C_s^{\top}\vec{k}_s= \bszero \in \ZZ_b^m\} ,
  \end{align*}
where, for $1\le j\le s$, we write $\vec{k}_j=(\kappa_{0,j},\kappa_{1,j},\dots)^{\top}\in \Zbinfty$ for $k_j$ with its $b$-adic expansion $k_j=\kappa_{0,j}+\kappa_{1,j}b+\cdots$, which is actually a finite expansion.
\end{definition}

Using the characters on $\Zbinfty^s$, the dual net $\mathcal{P}^{\perp}$ is also given by
  \begin{align*}
     \mathcal{P}^{\perp}=\{\bsk\in \nat_0^s \colon W_{\bsk}(\bsz) = 1\quad \text{for all $\bsz\in \Pcal$}\} .
  \end{align*}
A similar definition of the dual net has been introduced in \cite[Definition~2.5]{DM13} for digital nets in $\Zbinfty^s$ with finite digit expansions. Because of orthogonality of the characters, we have the following.
\begin{lemma}\label{lem:dual_Walsh}
Let $\mathcal{P}$ be a digital net in $\Zbinfty^s$, and let $\mathcal{P}^{\perp}$ be its dual net. Then we have
  \begin{align*}
     \sum_{\bsz\in \mathcal{P}} W_{\bsk}(\bsz) = \begin{cases}
     |\mathcal{P}| & \text{if $\bsk\in \mathcal{P}^{\perp}$},  \\
     0 & \text{otherwise} .
    \end{cases}
  \end{align*}
\end{lemma}

%%%%%%%%%%%%%%%%%%%%%%%%%%%%%%%%%%%%%%%%%%%%%%%%%%%%%%%%%%%
\subsection{Higher order polynomial lattice rules}
We define higher order polynomial lattice point sets as subsets of $\Zbinfty^s$, whose construction is based on rational functions over finite fields \cite{DP07}. Now let $b$ be a prime. We denote by $\ZZ_b[x]$ the set of all polynomials in $\ZZ_b$ and by $\ZZ_b((x^{-1}))$ the field of formal Laurent series in $\ZZ_b$. Every element of $\ZZ_b((x^{-1}))$ can be expressed in the form
  \begin{align*}
    L=\sum_{l=w}^{\infty}t_lx^{-l},
  \end{align*}
for some integer $w$ and $t_l\in \ZZ_b$. When $w>1$, we set $t_1=\cdots=t_{w-1}=0$. For $n\in \nat$, we define the mapping $v_n\colon \ZZ_b((x^{-1})) \to \Zbinfty$ by
  \begin{align*}
    v_n\left( \sum_{l=w}^{\infty}t_l x^{-l}\right) = (t_1,\ldots,t_n,0,0,\ldots)^{\top}.
  \end{align*}
We shall often identify an integer $n=n_0+n_1b+\cdots\in \nat_0$ with a polynomial $n(x)=n_0+n_1x+\cdots \in \ZZ_b[x]$. Then HOPLPSs are constructed as follows.

\begin{definition}\label{def:hopoly}
For $m, n,s \in \nat$ with $m\le n$ let $p \in \ZZ_b[x]$ with $\deg(p)=n$ and let $\bsq=(q_1,\ldots,q_s) \in (\ZZ_b[x])^s$. A higher order polynomial lattice point set (HOPLPS) $\Pcal(\bsq,p)\subset \Zbinfty^s$ is given by
  \begin{align*}
    \bsz_h &:= \left( v_{n}\left( \frac{h(x)q_1(x)}{p(x)} \right) , \ldots , v_{n}\left( \frac{h(x)q_s(x)}{p(x)} \right) \right) ,
  \end{align*}
for an integer $0\le h <b^m$, which is identified with $h(x)\in \ZZ_b[x]$. A QMC rule using $\pi(\Pcal(\bsq,p))$ is called a higher order polynomial lattice rule with modulus $p$ and generating vector $\bsq$.
\end{definition}

We define the truncated polynomial $\tr_n(k)$, associated with $k\in \nat_0$ whose $b$-adic expansion is given by $k=\kappa_0+\kappa_1 b+\cdots$, as
  \begin{align*}
     \tr_n(k)(x)=\kappa_0+\kappa_1 x+\dots +\kappa_{n-1} x^{n-1} .
  \end{align*}
Then the \emph{dual polynomial lattice} of a HOPLPS $\Pcal(\bsq,p)$ plays the same role as the dual net of a digital net. It is defined as follows.
\begin{definition}\label{def:hopoly_dual_net}
For $m,n,s\in \nat$ with $m\le n$, let $\Pcal(\bsq,p)$ be a HOPLPS. The dual polynomial lattice of $\Pcal(\bsq,p)$, denoted by $\Pcal^{\perp}(\bsq,p)$, is defined as
  \begin{align*}
     \Pcal^{\perp}(\bsq,p) := \{ & \bsk=(k_1,\dots,k_s)\in \nat_0^s \colon\\
                                        & \tr_n(k_1)q_1 + \dots + \tr_n(k_s)q_s\equiv a \pmod p\quad \text{with $\deg(a)<n-m$}\} .
  \end{align*}
\end{definition}
\noindent Accordingly, Lemma \ref{lem:dual_Walsh} is now replaced by the following lemma.
\begin{lemma}\label{lem:dual_poly_Walsh}
For $m, n,s\in \nat$ with $m\le n$, let $\Pcal(\bsq,p)$ be a HOPLPS, and let $\Pcal^{\perp}(\bsq,p)$ be its dual polynomial lattice. Then we have
   \begin{align*}
     \sum_{\bsz\in \Pcal(\bsq,p)}W_{\bsk}(\bsz) = \begin{cases}
     b^m & \text{if $\bsk\in \Pcal^{\perp}(\bsq,p)$},  \\
     0 & \text{otherwise} .
    \end{cases}
  \end{align*}
\end{lemma}

%%%%%%%%%%%%%%%%%%%%%%%%%%%%%%%%%%%%%%%%%%%%%%%%%%%%%%%%%%%
%%%%%%%%%%%%%%%%%%%%%%%%%%%%%%%%%%%%%%%%%%%%%%%%%%%%%%%%%%%
%%%%%%%%%%%%%%%%%%%%%%%%%%%%%%%%%%%%%%%%%%%%%%%%%%%%%%%%%%%
\section{Folded digital nets and integration error}\label{sec:wc_error}

%%%%%%%%%%%%%%%%%%%%%%%%%%%%%%%%%%%%%%%%%%%%%%%%%%%%%%%%%%%
\subsection{The $b$-adic tent transformation}
Let $b$ be an arbitrary positive integer greater than $1$ again. Here we describe the $b$-adic tent transformation ($b$-TT)
$\Phi_b \colon \Zbinfty \to \Zbinfty$.
Previously, the authors introduced the $b$-adic tent transformation
$\phi_b \colon [0,1] \to [0,1]$ in \cite{GSYxx}.
These two $b$-adic tent transformations satisfy
$\pi \circ \Phi_b \circ \sigma = \phi_b$.
Notice that the original (dyadic) tent transformation is given by $\phi_2(x)=1-|2x-1|$, see \cite{H02}.

The $b$-TT, which is denoted by $\Phi_b$, is defined as follows.
Let $z=(\zeta_1, \zeta_2, \cdots)^{\top} \in \Zbinfty$.
Then $\Phi_b \colon \Zbinfty \to \Zbinfty$ is given by
  \begin{align*}
     \Phi_b(z) := (\eta_1, \eta_2, \dots)^{\top} \quad \text{with}\quad \eta_i=\zeta_{i+1}-\zeta_1 \pmod b .
  \end{align*}
For a vector $\bsz = (z_1, \dots, z_s) \in \Zbinfty^s$, we define $\Phi_b(\bsz) := (\Phi_b(z_1), \dots, \Phi_b(z_s))$. 
We can see that $\Phi_b$ is a $\ZZ_b$-module homomorphism as well as a group homomorphism.

%%%%%%%%%%%%%%%%%%%%%%%%%%%%%%%%%%%%%%%%%%%%%%%%%%%%%%%%%%%
\subsection{Folded digital nets}
In the following we consider folded digital nets in $\Zbinfty^s$, that is, digital nets in $\Zbinfty^s$ that are folded using the $b$-TT. We recall that quadrature points in $[0,1]^s$ used for QMC integration are obtained by the mapping $\pi$. For the sake of completeness, we give the definition of folded digital nets below.
\begin{definition}\label{def:folded_digital_net}
Let $\mathcal{P}$ be a digital net in $\Zbinfty^s$. The folded digital net of $\mathcal{P}$, denoted by $\mathcal{P}_{\Phi_b}$, is defined as
  \begin{align*}
     \mathcal{P}_{\Phi_b}:=\{\Phi_b(\bsz) \colon \bsz\in \Pcal\}.
  \end{align*}
\end{definition}

\begin{remark}\label{rem:folded_digital_net}
Let $\mathcal{P}$ be a digital net in $\Zbinfty^s$ with generating matrices $C_1,\dots,C_s\in \ZZ_b^{\nat\times m}$. In view of Definition~\ref{def:digital_net}, the folded digital net of $\mathcal{P}$ coincides with the digital net in $\Zbinfty^s$ with generating matrices $TC_1,\dots,TC_s\in \ZZ_b^{\nat\times m}$, where the matrix $T=(t_{i,j})_{i,j\in \nat}$ is defined as
  \begin{align*}
     t_{i,j}=\begin{cases}
     b-1 & \text{if $j=1$}, \\
     1  & \text{if $j=i+1$}, \\
     0  & \text{otherwise}.     
     \end{cases}
  \end{align*}
This fact has already been utilized to study the $L_p$ discrepancy of two-dimensional folded Hammersley point sets \cite{Gxx3}.
\end{remark}

We now consider the dual net of folded digital nets. In order to give the following lemma, we need to introduce some notation. For $x\in \RR$, we denote by $\lfloor x \rfloor$ the unique integer $n$ satisfying the inequalities $n\le x<n+1$. For $\bsx=(x_1,\ldots,x_s)\in \RR^s$, we write $\lfloor \bsx \rfloor=(\lfloor x_1 \rfloor,\ldots,\lfloor x_s \rfloor)$ for short. For $k\in \nat$ whose $b$-adic expansion is given by $k = \kappa_0+\kappa_1b+\cdots$, let $\delta(k):=\kappa_0+\kappa_1+\cdots$, which is called the $b$-adic sum-of-digits of $k$. Moreover, we define 
  \begin{align*}
     \Ecal:=\{k\in \nat \colon \delta(k)\equiv 0 \pmod b\},
  \end{align*}
and $\Ecal_0:=\Ecal\cup \{0\}$.

\begin{lemma}\label{lem:folded_2}
Let $\mathcal{P}$ be a digital net in $\Zbinfty^s$, and let $\mathcal{P}_{\Phi_b}$ be its folded digital net. Further, let $\mathcal{P}^{\perp}$ be the dual net of $\mathcal{P}$, and let $\mathcal{P}^{\perp}_{\Phi_b}$ be the dual net of $\mathcal{P}_{\Phi_b}$. Then we have
  \begin{align*}
     \mathcal{P}^{\perp}_{\Phi_b} =\{\lfloor \bsk/b\rfloor \colon \bsk\in \Ecal_0^s \cap \mathcal{P}^{\perp}\} ,
  \end{align*}
where we write $\lfloor \bsk/b\rfloor=(\lfloor k_1/b\rfloor,\ldots,\lfloor k_s/b\rfloor)$ for $\bsk=(k_1,\ldots,k_s)$.
\end{lemma}

\begin{proof}
Let $\mathcal{P}$ be a digital net in $\Zbinfty^s$ with generating matrices $C_1,\ldots,C_s$. From Definition \ref{def:dual_net} and Remark~\ref{rem:folded_digital_net}, the dual net $\mathcal{P}^{\perp}_{\Phi_b}$ is given by
  \begin{align*}
     \mathcal{P}^{\perp}_{\Phi_b} & = \{\bsk\in \nat_0^s \colon (TC_1)^{\top}\vec{k}_1 \oplus \dots \oplus (TC_s)^{\top}\vec{k}_s = \bszero \in \ZZ_b^m \} \\
     & = \{\bsk\in \nat_0^s \colon C_1^{\top}T^{\top}\vec{k}_1 \oplus \dots \oplus C_s^{\top}T^{\top}\vec{k}_s = \bszero \in \ZZ_b^m \} .
  \end{align*}
For $1\le j\le s$, we have
  \begin{align*}
     T^{\top}\vec{k}_j = \left(
    \begin{array}{cccc}
      b-1 & b-1 & b-1 & \cdots \\
      1 & 0 & 0 & \cdots \\
      0 & 1 & 0 & \cdots \\
      0 & 0 & 1 & \cdots \\
      \vdots & \vdots & \vdots & \ddots 
    \end{array}\right) \left(
    \begin{array}{c}
      \kappa_{0,j} \\
      \kappa_{1,j} \\
      \kappa_{2,j} \\
      \kappa_{3,j} \\
      \vdots
    \end{array}\right) = \left(
    \begin{array}{c}
      -(\kappa_{0,j}+\kappa_{1,j}+\cdots) \\
      \kappa_{0,j} \\
      \kappa_{1,j} \\
      \kappa_{2,j} \\
      \vdots
    \end{array}\right) .
  \end{align*}
We now define $\beta:\nat_0\to \{0,\ldots,b-1\}$ by $\beta(k)=-\delta(k) \pmod b$ for $k\in \nat_0$. Then we have
\begin{align*}
     T^{\top}\vec{k}_j = \vec{l}_j \in \Zbinfty ,
  \end{align*}
where $l_j$ is given by 
  \begin{align*}
     l_j & = \beta(k_j)+\kappa_{0,j}b+\kappa_{1,j}b^2+\cdots \\
     & = \beta(k_j)+bk_j .
  \end{align*}
Using this notation, $\mathcal{P}^{\perp}_{\Phi_b}$ is given by
  \begin{align}\label{eq:fold_dual_proof}
     \mathcal{P}^{\perp}_{\Phi_b} & = \{\bsk\in \nat_0^s \colon C_1^{\top}\vec{l}_1 \oplus \dots \oplus C_s^{\top}\vec{l}_s = \bszero \in \ZZ_b^m \} \nonumber \\
     & = \{\bsk\in \nat_0^s \colon (l_1,\ldots,l_s) \in \mathcal{P}^{\perp} \} ,
  \end{align}
where the second equality stems from the definition of the dual net of $\mathcal{P}$, see Definition \ref{def:dual_net}. Here we see that 
  \begin{align*}
     \delta(l_j)\equiv -\delta(k_j)+\kappa_0+\kappa_1+\cdots \equiv -\delta(k_j)+\delta(k_j) \equiv 0 \pmod b .
  \end{align*}
This implies that $(l_1,\ldots,l_s)\in \Ecal_0^s$. Let us consider the equation $\delta(k'_j+bk_j)\equiv 0 \pmod b$ with a variable $k'_j\in \{0,\ldots,b-1\}$. It is easy to confirm that $k'_j=\beta(k_j)$ is the only solution of this equation. Moreover, since we have $\lfloor (k'_j+bk_j)/b \rfloor=k_j$ independently of the choice $k'_j\in \{0,\ldots,b-1\}$, $k_j$ in (\ref{eq:fold_dual_proof}) can be uniquely expressed as $\lfloor l_j/b\rfloor$. Thus we have
  \begin{align*}
     \mathcal{P}^{\perp}_{\Phi_b} = & \{\lfloor \bsl/b\rfloor \colon \bsl\in \Ecal_0^s \cap \mathcal{P}^{\perp} \} ,
  \end{align*}
which completes the proof.
\end{proof}

%%%%%%%%%%%%%%%%%%%%%%%%%%%%%%%%%%%%%%%%%%%%%%%%%%%%%%%%%%%
\subsection{Worst-case error in reproducing kernel Hilbert spaces}
Let us consider a reproducing kernel Hilbert space $\Hcal$ with reproducing kernel $\Kcal:[0,1]^s\times [0,1]^s\to \RR$. The inner product in $\Hcal$ is denoted by $\langle f,g \rangle_{\Hcal}$ for $f,g\in \Hcal$ and the associated norm is denoted by $\lVert f\rVert_{\Hcal}:=\sqrt{\langle f,f \rangle_{\Hcal}}$.

It is known that if a reproducing kernel $\Kcal$ satisfies $\int_{[0,1]^s}\sqrt{\Kcal(\bsx,\bsx)}\rd \bsx<\infty$ the squared worst-case error in the space $\Hcal$ of the QMC rule using a point set $P\subset [0,1]^s$ is given by
  \begin{align}\label{eq:worst-case_error}
     & \quad e^2(P,\Kcal) \nonumber \\
     & := \left( \sup_{\substack{f\in \Hcal \\ \|f\|_{\Hcal}\le 1}}|I(f)-Q(f;P)|\right)^2 \nonumber \\ 
     & = \int_{[0,1]^{2s}}\Kcal(\bsx,\bsy)\rd \bsx \rd \bsy-\frac{2}{|P|}\sum_{\bsx\in P}\int_{[0,1]^s}\Kcal(\bsx,\bsy)\rd \bsy+\frac{1}{|P|^2}\sum_{\bsx,\bsy\in P}\Kcal(\bsx,\bsy) ,
  \end{align}
while the squared initial error is given by
  \begin{align*}
     e^2(\emptyset,\Kcal) := \left( \sup_{\substack{f\in \Hcal\\ \|f\|_{\Hcal}\le 1}}|I(f)|\right)^2 = \int_{[0,1]^{2s}}\Kcal(\bsx,\bsy)\rd \bsx \rd \bsy .
  \end{align*}
For $\bsk, \bsl \in \nat_0^s$, the $(\bsk,\bsl)$-th Walsh coefficient of $\Kcal$ is given by
  \begin{align*}
     \hat{\Kcal}(\bsk,\bsl) := \int_{[0,1]^{2s}} \Kcal(\bsx,\bsy)\overline{\wal_{\bsk}(\bsx)}\wal_{\bsl}(\bsy)\rd \bsx \rd \bsy.
\end{align*}
We refer to \cite[Chapter~2]{DP10} for details. In the following we always assume $\int_{[0,1]^s}\sqrt{\Kcal(\bsx,\bsx)}\rd \bsx<\infty$ and consider the squared worst-case error of QMC rules using digital nets in $[0,1]^s$ in the space $\Hcal$.

\begin{proposition}\label{prop:worst_case_digital_net}
Let $\mathcal{P}, \mathcal{P}^{\perp}$ be a digital net in $\Zbinfty^s$ and its dual net, respectively, and let $\Kcal$ be a continuous reproducing kernel which satisfies $\int_{[0,1]^s}\sqrt{\Kcal(\bsx,\bsx)}\rd \bsx<\infty$.
We assume that
     $\sum_{\bsk,\bsl\in \nat_0^s}|\hat{\Kcal}(\bsk,\bsl)|<\infty$.
Then the squared worst-case error of QMC rules using $\pi(\mathcal{P})$ is given by
  \begin{align*}
     e^2(\pi(\mathcal{P}),\Kcal) = \sum_{\bsk,\bsl\in \mathcal{P}^{\perp}\setminus \{\bszero\}}\hat{\Kcal}(\bsk,\bsl) .
  \end{align*}
\end{proposition}

\begin{proof}
First we prove that for any $\bsz, \bsw \in \Zbinfty^s$ it holds that
\begin{align}\label{eq:Wexpansion}
\Kcal(\pi(\bsz), \pi(\bsw)) = \sum_{\bsk,\bsl\in \nat_0^s}\hat{\Kcal}(\bsk,\bsl)
W_{\bsk}(\bsz)\overline{W_{\bsl}(\bsw)}.
\end{align}
The proof of (\ref{eq:Wexpansion}) is similar to \cite[Proposition~A.20]{DP10}.
By the assumption that $\sum_{\bsk,\bsl\in \nat_0^s}|\hat{\Kcal}(\bsk,\bsl)|<\infty$,
$\sum_{\bsk,\bsl\in \nat_0^s}\hat{\Kcal}(\bsk,\bsl)
W_{\bsk}(\bsz)\overline{W_{\bsl}(\bsw)}$
converges absolutely. Therefore it suffices to show
that
$$
\lim_{n \to \infty} \sum_{\bsk,\bsl < b^n}\hat{\Kcal}(\bsk,\bsl)
W_{\bsk}(\bsz)\overline{W_{\bsl}(\bsw)}
= \Kcal(\pi(\bsz), \pi(\bsw)),
$$
where $\bsk = (k_1, \dots, k_s) < b^n$ means that $k_j < b^n$ holds
for every $j$.
Define two sets $H_n \subset \Zbinfty$ and $H(\bsz, \bsw, n) \subset \Zbinfty^{2s}$ as
$H_n := \{(\zeta'_1, \zeta'_2, \dots)^{\top} \in \Zbinfty \colon \zeta'_1 = \cdots = \zeta'_n = 0\}$
and
%$H(\bsz, \bsw, n) := \{(z'_1, \dots, z'_s, w'_1, \dots, w'_s) \in \Zbinfty^s
%\colon z_i \ominus z'_i \in H_n, w_i \ominus w'_i \in H_n \text{ for all } 1 \leq i \leq s\}$.
$H(\bsz, \bsw, n) := \{(\bsz', \bsw') \in \Zbinfty^{2s}
\colon \bsz \ominus \bsz' \in H_n^s, \bsw \ominus \bsw' \in H_n^s\}$.
Then we have
\begin{align*}
& \quad \sum_{\bsk,\bsl < b^n}\hat{\Kcal}(\bsk,\bsl) W_{\bsk}(\bsz)\overline{W_{\bsl}(\bsw)} \\
& =\sum_{\bsk,\bsl < b^n} W_{\bsk}(\bsz)\overline{W_{\bsl}(\bsw)}
\int_{[0,1]^{2s}} \Kcal(\bsx,\bsy)\overline{\wal_{\bsk}(\bsx)}\wal_{\bsl}(\bsy)\rd \bsx \rd \bsy \\
&=\sum_{\bsk,\bsl < b^n} W_{\bsk}(\bsz)\overline{W_{\bsl}(\bsw)}
\int_{[0,1]^{2s}} (\Kcal \circ \pi \circ \sigma)(\bsx,\bsy)\overline{W_{\bsk}(\sigma(\bsx))}W_{\bsl}(\sigma(\bsy))\rd \bsx \rd \bsy \\
&=\sum_{\bsk,\bsl < b^n} W_{\bsk}(\bsz)\overline{W_{\bsl}(\bsw)}
\int_{\Zbinfty^{2s}} (\Kcal \circ \pi) (\bsz',\bsw')\overline{W_{\bsk}(\bsz')}W_{\bsl}(\bsw') \rd \bsz' \rd \bsw' \\
&= \int_{\Zbinfty^{2s}} \Kcal(\pi(\bsz'), \pi(\bsw'))
\sum_{\bsk,\bsl < b^n} \overline{W_{\bsk}(\bsz')}W_{\bsl}(\bsw')W_{\bsk}(\bsz)\overline{W_{\bsl}(\bsw)} \rd \bsz' \rd \bsw'  \\
&= \int_{\Zbinfty^{2s}} \Kcal(\pi(\bsz'), \pi(\bsw'))
\sum_{\bsk < b^n}W_{\bsk}(\bsz \ominus \bsz')\sum_{\bsl < b^n}\overline{W_{\bsl}(\bsw \ominus \bsw')} \rd \bsz' \rd \bsw'  \\
&= b^{2sn}\int_{H(\bsz, \bsw, n)} \Kcal(\pi(\bsz'), \pi(\bsw'))
\rd \bsz' \rd \bsw',
\end{align*}
where we use Item~2 of Proposition~\ref{prop: map pi and rho}
and (\ref{eq:WandWalsh}),
Item~3 of Proposition~\ref{prop: map pi and rho} and Lemma~\ref{lem:Paley}
in the second, third and the last equality, respectively.
Since $\Kcal$ and $\pi$ are continuous, $\Kcal \circ \pi$ is also continuous.
Hence the last term of the above equality converges to
$\Kcal(\pi(\bsz), \pi(\bsw))$
as $n \to \infty$. 
This proves (\ref{eq:Wexpansion}).

Now we prove Proposition~\ref{prop:worst_case_digital_net}.
For the first term on the right-hand side of (\ref{eq:worst-case_error}), we have
  \begin{align*}
     \int_{[0,1]^{2s}}\Kcal(\bsx,\bsy)\rd \bsx \rd \bsy = \hat{\Kcal}(\bszero,\bszero) .
  \end{align*}
For the second term on the right-hand side of (\ref{eq:worst-case_error}), we have
\begin{align*}
     & \quad \frac{2}{|\mathcal{P}|}\sum_{\bsx\in \pi(\mathcal{P})}\int_{[0,1]^s}\Kcal(\bsx,\bsy)\rd \bsy \\
     & = \frac{1}{|\mathcal{P}|}\sum_{\bsx\in \pi(\mathcal{P})}\int_{[0,1]^s}\Kcal(\bsx,\bsy)\rd \bsy +\frac{1}{|\mathcal{P}|}\sum_{\bsx\in \pi(\mathcal{P})}\int_{[0,1]^s}\Kcal(\bsy,\bsx)\rd \bsy\\
     & = \frac{1}{|\mathcal{P}|}\sum_{\bsz\in \mathcal{P}}\int_{\Zbinfty^s}\Kcal(\pi(\bsz),\pi(\bsw))\rd \bsw +\frac{1}{|\mathcal{P}|}\sum_{\bsz\in \mathcal{P}}\int_{\Zbinfty^s}\Kcal(\pi(\bsw),\pi(\bsz))\rd \bsw\\
     & = \frac{1}{|\mathcal{P}|}\sum_{\bsz\in \mathcal{P}}\sum_{\bsk,\bsl\in \nat_0^s}\hat{\Kcal}(\bsk,\bsl)W_{\bsk}(\bsz)\int_{\Zbinfty^s}\overline{W_{\bsl}(\bsw)}\rd \bsw \\
     & \quad + \frac{1}{|\mathcal{P}|}\sum_{\bsz\in \mathcal{P}}\sum_{\bsk,\bsl\in \nat_0^s}\hat{\Kcal}(\bsk,\bsl)\overline{W_{\bsl}(\bsz)}\int_{\Zbinfty^s}W_{\bsk}(\bsw)\rd \bsw \\
     & = \sum_{\bsk\in \nat_0^s}\hat{\Kcal}(\bsk,\bszero) \frac{1}{|\mathcal{P}|}\sum_{\bsz\in \mathcal{P}}W_{\bsk}(\bsz)+\sum_{\bsl\in \nat_0^s}\hat{\Kcal}(\bszero,\bsl) \overline{\frac{1}{|\mathcal{P}|}\sum_{\bsz\in \mathcal{P}}W_{\bsl}(\bsz)} \\
     & = \sum_{\bsk\in \mathcal{P}^{\perp}}\hat{\Kcal}(\bsk,\bszero)+\sum_{\bsl\in \mathcal{P}^{\perp}}\hat{\Kcal}(\bszero,\bsl) ,
  \end{align*}
where we use the symmetry of $\Kcal$, Item~4 of Proposition~\ref{prop: map pi and rho}, (\ref{eq:Wexpansion}), Item~1 of Proposition~\ref{prop: infinite digit} and Lemma \ref{lem:dual_Walsh} in the first, second, third, fourth and fifth equalities, respectively. For the last term on the right-hand side of (\ref{eq:worst-case_error}), we have
  \begin{align*}
     \frac{1}{|\mathcal{P}|^2}\sum_{\bsx,\bsy\in \pi(\mathcal{P})}\Kcal(\bsx,\bsy)
& = \frac{1}{|\mathcal{P}|^2}\sum_{\bsz,\bsw\in \mathcal{P}}\Kcal(\pi(\bsz),\pi(\bsw))\\
& = \frac{1}{|\mathcal{P}|^2}\sum_{\bsz,\bsw\in \mathcal{P}}\sum_{\bsk,\bsl\in \nat_0^s}\hat{\Kcal}(\bsk,\bsl)W_{\bsk}(\bsz)\overline{W_{\bsl}(\bsw)} \\
& = \sum_{\bsk,\bsl\in \nat_0^s}\hat{\Kcal}(\bsk,\bsl)\frac{1}{|\mathcal{P}|}\sum_{\bsz\in \mathcal{P}}W_{\bsk}(\bsz)\overline{\frac{1}{|\mathcal{P}|}\sum_{\bsw\in \mathcal{P}}W_{\bsl}(\bsw)} \\
& = \sum_{\bsk,\bsl\in \mathcal{P}^{\perp}}\hat{\Kcal}(\bsk,\bsl) ,
  \end{align*}
where we use (\ref{eq:Wexpansion}) and Lemma~\ref{lem:dual_Walsh} again in the second  and forth equality, respectively. Substituting these results into the right-hand side of (\ref{eq:worst-case_error}), the result follows.
\end{proof}

As mentioned in Remark~\ref{rem:folded_digital_net}, when $\mathcal{P}$ is a digital net in $\Zbinfty^s$, its folded digital net $\mathcal{P}_{\Phi_b}$ is also a digital net in $\Zbinfty^s$. Therefore, Proposition \ref{prop:worst_case_digital_net} can be applied to QMC rules using folded digital nets. We have the following results. 

\begin{proposition}\label{prop:worst_case_folded_digital_net}
Let $\mathcal{P}, \mathcal{P}_{\Phi_b}$ be a digital net in $\Zbinfty^s$ and its folded digital net, respectively, and let $\mathcal{P}^{\perp}$ be the dual net of $\mathcal{P}$. Further, let $\Kcal$ be a continuous reproducing kernel which satisfies $\int_{[0,1]^s}\sqrt{\Kcal(\bsx,\bsx)}\rd \bsx<\infty$. Assume that $\sum_{\bsk,\bsl\in \nat_0^s}|\hat{\Kcal}(\bsk,\bsl)|<\infty$. Then the squared worst-case error of QMC rules using $\pi(\mathcal{P}_{\Phi_b})$ is given by
  \begin{align*}
     e^2(\pi(\mathcal{P}_{\Phi_b}),\Kcal) = \sum_{\bsk,\bsl\in \Ecal_0^s \cap \mathcal{P}^{\perp}\setminus \{\bszero\}}\hat{\Kcal}(\lfloor \bsk/b\rfloor,\lfloor \bsl/b\rfloor) .
  \end{align*}
\end{proposition}

\begin{proof}
Since $\mathcal{P}_{\Phi_b}$ is a digital net in $\Zbinfty^s$, we have from Proposition \ref{prop:worst_case_digital_net}
  \begin{align*}
     e^2(\pi(\mathcal{P}_{\Phi_b}),\Kcal) = \sum_{\bsk,\bsl\in \mathcal{P}_{\Phi_b}^{\perp}\setminus \{\bszero\}}\hat{\Kcal}(\bsk,\bsl) ,
  \end{align*}
where $\mathcal{P}_{\Phi_b}^{\perp}$ is the dual net of $\mathcal{P}_{\Phi_b}$. Applying Lemma \ref{lem:folded_2} to the right-hand side, we have the result.
\end{proof}

%%%%%%%%%%%%%%%%%%%%%%%%%%%%%%%%%%%%%%%%%%%%%%%%%%%%%%%%%%%
%%%%%%%%%%%%%%%%%%%%%%%%%%%%%%%%%%%%%%%%%%%%%%%%%%%%%%%%%%%
%%%%%%%%%%%%%%%%%%%%%%%%%%%%%%%%%%%%%%%%%%%%%%%%%%%%%%%%%%%
\section{Bound on the worst-case error in $\Hcal_{\alpha}$}\label{sec:bound}
\subsection{Unanchored Sobolev spaces $\Hcal_{\alpha}$}
First we follow the expositions of \cite{BD09,DP07} to introduce the reproducing kernel Hilbert space $\Hcal_{\alpha}$ that we consider in the remainder of this paper.

Let $\alpha$ be a positive integer greater than 1. For the one-dimensional unweighted case, the inner product of $\Hcal_{\alpha}$ is given by
  \begin{align*}
    \langle f,g\rangle_{\Hcal_{\alpha}} = \sum_{\tau=0}^{\alpha-1}\int_{0}^{1}f^{(\tau)}(x)\rd x\int_{0}^{1}g^{(\tau)}(x)\rd x + \int_{0}^{1}f^{(\alpha)}(x)g^{(\alpha)}(x)\rd x ,
  \end{align*}
where $f^{(\tau)}$ denotes the $\tau$-th derivative of $f$ and $f^{(0)}=f$. The reproducing kernel is $1+ \Kcal_{\alpha,(1)}(\cdot,\cdot)$ for the function $\Kcal_{\alpha,(1)}: [0,1]\times [0,1]\to \RR$ defined as
  \begin{align*}
    \Kcal_{\alpha,(1)}(x,y) := \sum_{\tau=1}^{\alpha}\frac{\Bcal_{\tau}(x)\Bcal_{\tau}(y)}{(\tau !)^2}+(-1)^{\alpha+1}\frac{\Bcal_{2\alpha}(|x-y|)}{(2\alpha)!} ,
  \end{align*}
for $x,y\in [0,1]$, where $\Bcal_{\tau}$ denotes the Bernoulli polynomial of degree $\tau$.

Let us consider the $s$-dimensional weighted case. Let $\gamma_1,\ldots,\gamma_s$ be non-negative real numbers called \emph{weights}, which moderate the relative importance of different variables in the space $\Hcal_{\alpha}$. For a vector $(\alpha_1,\dots,\alpha_s)\in \nat_0^s$ with $0\le \alpha_j\le \alpha$ for $1\le j\le s$, we denote by $f^{(\alpha_1,\dots,\alpha_s)}$ the partial mixed derivative of $f$ of $(\alpha_1,\dots,\alpha_s)$-th order. Now the inner product of $\Hcal_{\alpha}$ is given by
  \begin{align*}
    \langle f,g\rangle_{\Hcal_{\alpha}} & = \sum_{u\subseteq \{1,\ldots,s\}}\left(\prod_{j\in u}\gamma_j^{-1}\right)\sum_{v\subseteq u}\sum_{\bstau_{u\setminus v}\in \{1,\ldots, \alpha-1\}^{|u\setminus v|}}\\
    & \quad \int_{[0,1]^{|v|}} \left(\int_{[0,1]^{s-|v|}}f^{(\bstau_{u\setminus v},\bsalpha_v,\bszero)}(\bsx)\rd\bsx_{-v}\right) \\
    & \quad \times \left(\int_{[0,1]^{s-|v|}}g^{(\bstau_{u\setminus v},\bsalpha_v,\bszero)}(\bsx)\rd\bsx_{-v}\right)\rd\bsx_{v} ,
  \end{align*}
where we use the following notation: For $\bstau_{u\setminus v}=(\tau_j)_{j\in u\setminus v}$, we denote by $(\bstau_{u\setminus v},\bsalpha_v,\bszero)$ the vector in which the $j$-th component is $\tau_j$ for $j\in u\setminus v$, $\alpha$ for $j\in v$, and 0 for $\{1,\ldots,s\}\setminus u$. For $v\subseteq \{1,\ldots,s\}$, we simply write $-v:=\{1,\ldots,s\}\setminus v$, $\bsx_{v}=(x_j)_{j\in v}$ and $\bsx_{-v}=(x_j)_{j\in -v}$. Further for $u\subseteq \{1,\ldots,s\}$ such that $\gamma_j=0$ for some $j\in u$, we assume that the corresponding inner double sum equals 0 and we set $0/0=0$. The reproducing kernel is given by
  \begin{align*}
    \Kcal_{\alpha}(\bsx,\bsy) = \prod_{j=1}^{s}\left( 1+\gamma_j \Kcal_{\alpha,(1)}(x_j,y_j)\right) ,
  \end{align*}
for $\bsx=(x_1,\ldots,x_s),\bsy=(y_1,\ldots,y_s)\in [0,1]^s$. Here we call $\Hcal_{\alpha}$ a weighted unanchored Sobolev space of smoothness of order $\alpha$.

We now consider a Walsh series expansion of $\Kcal_{\alpha}$,
  \begin{align*}
     & \quad \sum_{\bsk,\bsl\in \nat_0^s}\hat{\Kcal}_{\alpha}(\bsk,\bsl)\wal_{\bsk}(\bsx)\overline{\wal_{\bsl}(\bsy)} \\
     & = \sum_{u,v\subseteq \{1,\dots,s\}}\sum_{\bsk_u\in \nat^{|u|}}\sum_{\bsl_v\in \nat^{|v|}}\hat{\Kcal}_{\alpha}((\bsk_u,\bszero),(\bsl_v,\bszero))\wal_{(\bsk_u,\bszero)}(\bsx)\overline{\wal_{(\bsl_v,\bszero)}(\bsy)},
  \end{align*}
where we denote by $(\bsk_u,\bszero)$ the $s$-dimensional vector whose $j$-th component is $k_j$ if $j\in u$ and zero otherwise, and we use the same notation for $(\bsl_v,\bszero)$. According to \cite[Lemma~14]{BD09}, for $u,v\subseteq \{1,\dots,s\}$, $\bsk_u\in \nat^{|u|}$ and $\bsl_v\in \nat^{|v|}$, the $((\bsk_u,\bszero),(\bsl_v,\bszero))$-th Walsh coefficient is given by
  \begin{align}\label{eq:walsh_series_Sobolev}
     \hat{\Kcal}_{\alpha}((\bsk_u,\bszero),(\bsl_v,\bszero)) = & \begin{cases}
     \gamma_u \prod_{j\in u}\hat{\Kcal}_{\alpha,(1)}(k_j,l_j) & \text{if $u=v$},  \\
     0 & \text{otherwise} .
    \end{cases}
  \end{align}
Here we write $\gamma_u=\prod_{j\in u}\gamma_j$ for $u\subseteq \{1,\ldots,s\}$, where the empty product equals $1$. A bound on the Walsh coefficients $\hat{\Kcal}_{\alpha,(1)}(k,l)$ for $k,l\in \nat$ is given in \cite[Section~3]{BD09} by
  \begin{align}\label{eq:walsh_coeff_Sobolev}
     \left| \hat{\Kcal}_{\alpha,(1)}(k,l) \right| \le D_{\alpha,b}b^{-\mu_{\alpha}(k)-\mu_{\alpha}(l)}  ,
  \end{align}
where $D_{\alpha,b}$ is positive and depends only on $\alpha$ and $b$, which can be calculated explicitly, and $\mu_{\alpha}(k)$ is defined for $k\in \nat$ as
  \begin{align*}
     \mu_{\alpha}(k)=a_1+\dots+a_{\min(v,\alpha)},
  \end{align*}
where we denote the $b$-adic expansion of $k$ by $k=\kappa_1b^{a_1-1}+\dots+\kappa_vb^{a_v-1}$ with $0< \kappa_1,\dots,\kappa_v<b$ and $a_1>\dots>a_v>0$.

Here $\Kcal_{\alpha}$ is continuous in the usual topology of $[0,1]^s$. Therefore, we have $\int_{[0,1]^s}\sqrt{\Kcal_{\alpha}(\bsx,\bsx)}\rd \bsx<\infty$, and $\sum_{\bsk,\bsl \in \nat_0^s} |\hat{\Kcal}_{\alpha}(\bsk, \bsl)|$ is also finite since
  \begin{align*}
     \sum_{\bsk,\bsl \in \nat_0^s} |\hat{\Kcal}_{\alpha}(\bsk, \bsl)| & \le \sum_{u\subseteq \{1,\dots,s\}} \gamma_uD^{|u|}_{\alpha,b}\sum_{\bsk_u,\bsl_u\in \nat^{|u|}}\prod_{j\in u}b^{-\mu_{\alpha}(k_j)-\mu_{\alpha}(l_j)} \\
     & = \sum_{u\subseteq \{1,\dots,s\}} \prod_{j\in u}\gamma_jD_{\alpha,b}\left(\sum_{k=1}^{\infty}b^{-\mu_{\alpha}(k)}\right)^2 \\
     & = \prod_{j=1}^{s}\left[ 1+\gamma_jD_{\alpha,b}\left( \sum_{k=1}^{\infty}b^{-\mu_{\alpha}(k)}\right)^2\right] ,
  \end{align*}
where the inner sum is shown to be finite as in \cite[Lemma~4.2]{DP07}. %From the continuity of $\Kcal_{\alpha}$ and the fact that $\sum_{\bsk,\bsl \in \nat_0^s} |\hat{\Kcal}_{\alpha}(\bsk, \bsl)|$ is finite, $\Kcal_{\alpha}$ satisfies the sufficient condition described in Proposition \ref{prop:pointwise_convergence}. Thus, the Walsh series of $\Kcal_{\alpha}$ converges to $\Kcal_{\alpha}$ pointwise absolutely.
Thus we can apply Proposition \ref{prop:worst_case_folded_digital_net} to the reproducing kernel $\Kcal_{\alpha}$.

\subsection{Bound on the worst-case error}
Now we show that the worst-case error in $\Hcal_{\alpha}$ of QMC rules using $\pi(\mathcal{P}_{\Phi_b})$ is bounded from above as follows.

\begin{theorem}\label{thm:bound_worst_case}
Let $\mathcal{P}, \mathcal{P}_{\Phi_b}$ be a digital net in $\Zbinfty^s$ and its folded digital net, respectively, and let $\mathcal{P}^{\perp}$ be the dual net of $\mathcal{P}$. Then the worst-case error of QMC rules using $\pi(\mathcal{P}_{\Phi_b})$ in $\Hcal_{\alpha}$ is bounded by
  \begin{align*}
     e(\pi(\mathcal{P}_{\Phi_b}),\Kcal_{\alpha}) \le \sum_{\emptyset \ne u\subseteq \{1,\ldots,s\}}\gamma_u^{1/2} D_{\alpha,b}^{|u|/2}\sum_{\substack{\bsk_u\in \Ecal^{|u|}\\ (\bsk_u,\bszero) \in \mathcal{P}^{\perp}}}b^{-\mu_{\alpha}(\lfloor \bsk_u/b\rfloor)} ,
  \end{align*}
where we write $\mu_{\alpha}(\lfloor \bsk_u/b\rfloor)=\sum_{j\in u}\mu_{\alpha}(\lfloor k_j/b\rfloor)$.
\end{theorem}

\begin{proof}
Since $\Kcal_{\alpha}$ satisfies all the assumptions in Proposition \ref{prop:worst_case_folded_digital_net}, we can use the result of Proposition \ref{prop:worst_case_folded_digital_net}, (\ref{eq:walsh_series_Sobolev}) and (\ref{eq:walsh_coeff_Sobolev}), in this order, to obtain
  \begin{align*}
     e^2(\pi(\mathcal{P}_{\Phi_b}),\Kcal_{\alpha}) & = \sum_{\bsk,\bsl\in \Ecal_0^s \cap \mathcal{P}^{\perp}\setminus \{\bszero\}}\hat{\Kcal}(\lfloor \bsk/b\rfloor,\lfloor \bsl/b\rfloor) \\
     & = \sum_{\emptyset \ne u\subseteq \{1,\ldots,s\}}\sum_{\substack{\bsk_u,\bsl_u\in \Ecal^{|u|}\\  (\bsk_u,\bszero),(\bsl_u,\bszero)\in \mathcal{P}^{\perp}}}\hat{\Kcal}((\lfloor \bsk_u/b\rfloor,\bszero),(\lfloor \bsl_u/b\rfloor,\bszero)) \\
     & \le \sum_{\emptyset \ne u\subseteq \{1,\ldots,s\}}\gamma_u D_{\alpha,b}^{|u|}\sum_{\substack{\bsk_u,\bsl_u\in \Ecal^{|u|}\\  (\bsk_u,\bszero),(\bsl_u,\bszero)\in \mathcal{P}^{\perp}}}b^{-\mu_{\alpha}(\lfloor \bsk_u/b\rfloor)-\mu_{\alpha}(\lfloor \bsl_u/b\rfloor)} \\
     & = \sum_{\emptyset \ne u\subseteq \{1,\ldots,s\}}\gamma_u D_{\alpha,b}^{|u|}\left(\sum_{\substack{\bsk_u\in \Ecal^{|u|}\\  (\bsk_u,\bszero)\in \mathcal{P}^{\perp}}}b^{-\mu_{\alpha}(\lfloor \bsk_u/b\rfloor)}\right)^2 \\
     & \le \left(\sum_{\emptyset \ne u\subseteq \{1,\ldots,s\}}\gamma_u^{1/2} D_{\alpha,b}^{|u|/2}\sum_{\substack{\bsk_u\in \Ecal^{|u|}\\  (\bsk_u,\bszero)\in \mathcal{P}^{\perp}}}b^{-\mu_{\alpha}(\lfloor \bsk_u/b\rfloor)}\right)^2 .
  \end{align*}
Thus the result follows.
\end{proof}

%%%%%%%%%%%%%%%%%%%%%%%%%%%%%%%%%%%%%%%%%%%%%%%%%%%%%%%%%%%
%%%%%%%%%%%%%%%%%%%%%%%%%%%%%%%%%%%%%%%%%%%%%%%%%%%%%%%%%%%
%%%%%%%%%%%%%%%%%%%%%%%%%%%%%%%%%%%%%%%%%%%%%%%%%%%%%%%%%%%
\section{Component-by-component construction}\label{sec:cbc}

In this section, we investigate the component-by-component (CBC) construction as an explicit means of constructing good FHOPLPSs. More precisely, we attempt to find good HOPLPSs $\Pcal(\bsq,p)$ by using the CBC construction algorithm, such that QMC rules using those folded point sets (FHOPLPSs) in $[0,1]^s$, denoted by $\pi(\Pcal_{\Phi_b}(\bsq,p))$, achieve a good convergence rate of the worst-case error in $\Hcal_{\alpha}$. As above, we write $$\Pcal_{\Phi_b}(\bsq,p):=\{\Phi_b(\bsz)\colon \bsz\in \Pcal(\bsq,p)\}.$$ For this purpose, we employ the bound of the worst-case error shown in Theorem \ref{thm:bound_worst_case} as a quality criterion of $\Pcal(\bsq,p)$. In order to emphasize the role of a modulus $p\in \ZZ_b[x]$ and a generating vector $\bsq\in (\ZZ_b[x])^s$ of HOPLPSs, we simply write the bound on $e^2(\pi(\Pcal_{\Phi_b}(\bsq,p)),\Kcal_{\alpha})$ as
  \begin{align*}
     B_{\alpha}(\bsq,p) = \sum_{\emptyset \ne u\subseteq \{1,\ldots,s\}}\gamma_u^{1/2}D_{\alpha,b}^{|u|/2}\sum_{\substack{\bsk_u\in \Ecal^{|u|}\\ (\bsk_u,\bszero) \in \Pcal^{\perp}(\bsq,p)}}b^{-\mu_{\alpha}(\lfloor \bsk_u/b\rfloor)} .
  \end{align*}

In the following let $b$ be a prime. We write $\bsq_{\tau}=(q_1,\ldots,q_{\tau})\in (\ZZ_b[x])^{\tau}$ for $\tau\in \nat_0$, where $\bsq_{0}$ denotes the empty set, and define
  \begin{align*}
    B_{\alpha}(\bsq_{\tau},p) := \sum_{\emptyset \ne u\subseteq \{1,\ldots,\tau\}}\gamma_u^{1/2}D_{\alpha,b}^{|u|/2}\sum_{\substack{\bsk_u\in \Ecal^{|u|}\\ (\bsk_u,\bszero) \in \Pcal^{\perp}(\bsq_{\tau},p)}}b^{-\mu_{\alpha}(\lfloor \bsk_u/b\rfloor)} ,
  \end{align*}
for $1\le \tau\le s$, where we denote by $(\bsk_u,\bszero)$ the $\tau$-dimensional vector in which the $j$-th component is $k_j$ for $j\in u$ and 0 for $j\in \{1,\ldots,\tau\}\setminus u$. In view of Definition \ref{def:hopoly}, we can restrict ourselves to considering each $q_j\in \ZZ_b[x]$ such that $\deg(q_j)<n$ without loss of generality. We write
  \begin{align*}
    R_{b,n} := \{q\in \ZZ_b[x] \colon \deg(q)<n \} .
  \end{align*}
Then the CBC construction proceeds as follows.

\begin{algorithm}\label{algorithm:cbc}
Let $b$ be a prime. For $s,m,n,\alpha\in \nat$ with $\alpha\ge 2$ and $n\ge m$, do the following:
	\begin{enumerate}
		\item Choose an irreducible polynomial $p\in \ZZ_b[x]$ such that $\deg(p)=n$.
		\item For $\tau=1,\ldots, s$, assume that $\bsq_{\tau-1}$ have already been found. Choose $q_{\tau}\in R_{b,n}$ which minimizes $B_{\alpha}((\bsq_{\tau-1},\tilde{q}_{\tau}),p)$ as a function of $\tilde{q}_{\tau}$.
	\end{enumerate}
\end{algorithm}

The following theorem gives an upper bound on the worst-case error in $\Hcal_{\alpha}$ of FHOPLPS for $\bsq$ and $p$ that are found according to Algorithm \ref{algorithm:cbc}.

\begin{theorem}\label{theorem:cbc}
Let $b$ be a prime. For $1\le \tau\le s$, let $p\in \ZZ_b[x]$ and $\bsq_{\tau}\in R_{b,n}^{\tau}$ be found according to Algorithm \ref{algorithm:cbc}. Then we have
  \begin{align*}
    B_{\alpha}(\bsq_{\tau},p) \le \frac{1}{b^{\min( m/\lambda, 2n)}}\left[-1+\prod_{j=1}^{\tau}\left(1+\gamma_j^{\lambda/2} D_{\alpha,b}^{\lambda/2}A_{\alpha,b,\lambda}\right)\right]^{1/\lambda} ,
  \end{align*}
for any $1/\alpha< \lambda \le 1$, where $A_{\alpha,b,\lambda}$ is positive and depends only on $\alpha,b$ and $\lambda$.
\end{theorem}
\noindent
The proof of this theorem is given in Appendix~\ref{appendix:a}.

\begin{remark}\label{remark:cbc}
Let $p\in \ZZ_b[x]$ and $\bsq\in R_{b,n}^s$ be found according to Algorithm \ref{algorithm:cbc}. When $n\ge \alpha m/2$, we have $\min(m/\lambda, 2n)=m/\lambda$ so that
  \begin{align*}
    B_{\alpha}(\bsq,p) \le \frac{1}{b^{m/\lambda}}\left[-1+\prod_{j=1}^{s}\left(1+\gamma_j^{\lambda/2}D_{\alpha,b}^{\lambda/2}A_{\alpha,b,\lambda}\right)\right]^{1/\lambda} ,
  \end{align*}
for $1/\alpha<\lambda \le 1$. Since we cannot achieve the convergence rate of the worst-case error of order $b^{-\alpha m}$ in $\Hcal_{\alpha}$ \cite{S63}, our result is optimal. The degree of the modulus required to achieve the optimal rate of the worst-case error is reduced by half as compared to that obtained in \cite[Theorem~1]{BDLNP12}.
\end{remark}

%%%%%%%%%%%%%%%%%%%%%%%%%%%%%%%%%%%%%%%%%%%%%%%%%%%%%%%%%%%
%%%%%%%%%%%%%%%%%%%%%%%%%%%%%%%%%%%%%%%%%%%%%%%%%%%%%%%%%%%
%%%%%%%%%%%%%%%%%%%%%%%%%%%%%%%%%%%%%%%%%%%%%%%%%%%%%%%%%%%
\section{Fast construction algorithm}\label{sec:fast}
Finally, in this section, we discuss how to calculate $B_{\alpha}(\bsq,p)$ efficiently and how to obtain the fast CBC construction using the fast Fourier transform.

%%%%%%%%%%%%%%%%%%%%%%%%%%%%%%%%%%%%%%%%%%%%%%%%%%%%%%%%%%%
\subsection{Efficient calculation of the quality criterion}
From Lemma \ref{lem:folded_2} and the definition of $B_{\alpha}(\bsq,p)$, we have
  \begin{align*}
     B_{\alpha}(\bsq,p) = \sum_{\emptyset \ne u\subseteq \{1,\ldots,s\}}\gamma_u^{1/2}D_{\alpha,b}^{|u|/2}\sum_{\substack{\bsk_u\in \nat^{|u|}\\ (\bsk_u,\bszero) \in \Pcal_{\Phi_b}^{\perp}(\bsq,p)}}b^{-\mu_{\alpha}(\bsk_u)} ,
  \end{align*}
where $\mathcal{P}_{\Phi_b}^{\perp}(\bsq,p)$ is the dual net of $\mathcal{P}_{\Phi_b}(\bsq,p)$. Using Lemma \ref{lem:dual_Walsh}, we have
  \begin{align*}
     B_{\alpha}(\bsq,p) & = \sum_{\emptyset \ne u\subseteq \{1,\ldots,s\}}\gamma_u^{1/2}D_{\alpha,b}^{|u|/2}\sum_{\bsk_u\in \nat^{|u|}}b^{-\mu_{\alpha}(\bsk_u)}\frac{1}{b^m}\sum_{\bsz\in \Pcal_{\Phi_b}(\bsq,p)}W_{(\bsk_u,\bszero)}(\bsz) \\
     & = \sum_{\emptyset \ne u\subseteq \{1,\ldots,s\}}\gamma_u^{1/2}D_{\alpha,b}^{|u|/2}\sum_{\bsk_u\in \nat^{|u|}}b^{-\mu_{\alpha}(\bsk_u)}\frac{1}{b^m}\sum_{\bsz\in \Pcal(\bsq,p)}W_{(\bsk_u,\bszero)}(\Phi_b(\bsz)) \\
     & = \frac{1}{b^m}\sum_{\bsz\in \Pcal(\bsq,p)}\sum_{\emptyset \ne u\subseteq \{1,\ldots,s\}}\gamma_u^{1/2}D_{\alpha,b}^{|u|/2}\sum_{\bsk_u\in \nat^{|u|}}b^{-\mu_{\alpha}(\bsk_u)}W_{(\bsk_u,\bszero)}(\Phi_b(\bsz)) \\
     & = \frac{1}{b^m}\sum_{\bsz\in \Pcal(\bsq,p)}\sum_{\emptyset \ne u\subseteq \{1,\ldots,s\}}\prod_{j\in u}\gamma_j^{1/2}D_{\alpha,b}^{1/2}\sum_{k_j=1}^{\infty}b^{-\mu_{\alpha}(k_j)}W_{k_j}(\Phi_b(z_j)) \\
     & = -1+\frac{1}{b^m}\sum_{\bsz\in \Pcal(\bsq,p)}\prod_{j=1}^{s}\left[ 1+\gamma_j^{1/2}D_{\alpha,b}^{1/2}\chi_b\circ \Phi_b(z_j)\right] ,
  \end{align*}
where we define the function $\chi_b:\Zbinfty \to \RR$ by
  \begin{align*}
     \chi_b(z) := \sum_{k=1}^{\infty}b^{-\mu_{\alpha}(k)}W_{k}(z) .
  \end{align*}
The difficulty in calculating $B_{\alpha}(\bsq,p)$ lies in the fact that $\chi_b(z)$ is expressed as an infinite sum over $k\in \nat$. Under the assumption that $z\in G$ is given in the form $$(\zeta_1,\zeta_2,\ldots,\zeta_n,0,0,\ldots)^{\top},$$ for $n \in \nat$ and $\zeta_i\in \ZZ_b$, $1\le i\le n$, it is possible to calculate $\chi_b(z)$ in at most $O(\alpha n)$ arithmetic operations as in \cite[Theorem~2]{BDLNP12}. In the following we show that it is also possible to calculate $\chi_b\circ \Phi_b(z)$ in at most $O(\alpha n)$ arithmetic operations under the same assumption on $z$. Notice that this assumption on $z$ is natural from Definition \ref{def:hopoly}.

\begin{theorem}\label{theorem:chi_calculation}
Let $z\in \Zbinfty$ be given in the form $(\zeta_1,\zeta_2,\ldots,\zeta_n,0,0,\ldots)^{\top}$ for $n \in \nat$ and $\zeta_i\in \ZZ_b$, $1\le i\le n$, that is, $\zeta_{n+1}=\zeta_{n+2}=\cdots=0$. Then $\chi_b\circ \Phi_b(z)$ can be calculated as follows: we define the following vectors
  \begin{align*}
    \bsU(\zeta_1) & = (U_0(\zeta_1),U_1(\zeta_1),\ldots,U_{\alpha-1}(\zeta_1)) ,\\
    \tilde{\bsU}(\zeta_1) & = (\tilde{U}_0(\zeta_1),\tilde{U}_1(\zeta_1),\ldots,\tilde{U}_{\alpha-1}(\zeta_1)) ,\\
    \bsV(z) & = (V_1(z),\ldots,V_{\alpha-1}(z)) ,\\
    \tilde{\bsV}(z) & = (\tilde{V}_{\alpha}(z),\ldots,\tilde{V}_1(z)) ,
  \end{align*}
where we set
  \begin{align*}
    U_0(\zeta_1) & = 1 ,\\
    U_t(\zeta_1) & = \frac{1}{b^{t(n-1)}}\prod_{i=1}^{t}\frac{\rho(\zeta_1)}{b^i-1}\quad \text{for $1\le t\le \alpha-1$},\\
    \tilde{U}_t(\zeta_1) & = \sum_{v=t}^{\alpha-1}U_{v-t}(\zeta_1)\quad \text{for $0\le t\le \alpha-1$},
  \end{align*}
where
  \begin{align*}
    \rho(\zeta_1) = \begin{cases}
    b-1 & \text{if $\zeta_1=0$},  \\
    -1 & \text{otherwise} .
    \end{cases}
  \end{align*}
We further define
  \begin{align*}
    V_t(z) & = \sum_{0<a_t<\cdots<a_1<n}\prod_{i=1}^{t}b^{-a_i}L(z,a_i+1) ,\\
    \tilde{V}_t(z) & = \sum_{0<a_t<\cdots<a_1<n}b^{a_t-1}[\Phi_b(z)\in H_{a_t-1}] \prod_{i=1}^{t}b^{-a_i}L(z,a_i+1) ,
  \end{align*}
for $1\le t\le \alpha-1$ and $1\le t\le \alpha$, respectively. Here the value of $[\Phi_b(z)\in H_{a_t-1}]$ equals 1 if $\Phi_b(z)\in H_{a_t-1}$, and $0$ otherwise, where $H_{a_t-1}:=\{(\zeta'_1,\zeta'_2,\dots)^{\top}\in \Zbinfty\colon \zeta'_1=\cdots=\zeta'_{a_t-1}=0\}$ for $a_t\in \nat$. Furthermore, $L(z,a_i+1)$ is defined as
  \begin{align*}
    L(z,a_i+1) := \begin{cases}
     b-1 & \text{if $\zeta_{a_i+1}=\zeta_{1}$},  \\
     -1  & \text{if $\zeta_{a_i+1}\ne \zeta_{1}$}.  \\
    \end{cases}
  \end{align*}
Using this notation, we have the following:
\begin{itemize}
\item If $\zeta_1=\cdots = \zeta_n=0$,
  \begin{align*}
    \chi_b\circ \Phi_b(z) = \sum_{v=1}^{\alpha-1}\prod_{i=1}^{v}\frac{b-1}{b^i-1}+\frac{b^{\alpha}-1}{b^{\alpha}-b}\prod_{i=1}^{\alpha}\frac{b-1}{b^i-1} .
  \end{align*}
\item Otherwise,
  \begin{align*}
    \chi_b\circ \Phi_b(z) = \tilde{\bsU}_{1:\alpha-1}(\zeta_1)\cdot \bsV(z)+(\tilde{U}_{0}(\zeta_1)-1)+\bsU(\zeta_1)\cdot \tilde{\bsV}(z) ,
  \end{align*}
where $\bsa\cdot \bsb$ denotes the dot product and $\tilde{\bsU}_{1:\alpha-1}(\zeta_1)$ is the vector of the last $\alpha-1$ components of $\tilde{\bsU}(\zeta_1)$.
\end{itemize}
\end{theorem}
\noindent
The proof of this theorem is given in Appendix~\ref{appendix:b}.

\begin{remark}
Note that $V_t(z)$ and $\tilde{V}_t(z)$ can be calculated as
  \begin{align*}
    V_t(z) = \sum_{a_1=t}^{n-1}b^{-a_1}L(z,a_1+1)\sum_{a_2=t-1}^{a_1-1}b^{-a_2}L(z,a_2+1)\cdots \sum_{a_t=1}^{a_{t-1}-1}b^{-a_t}L(z,a_t+1) ,
  \end{align*}
and
  \begin{align*}
    \tilde{V}_t(z) & = \sum_{a_1=t}^{n-1}b^{-a_1}L(z,a_1+1)\sum_{a_2=t-1}^{a_1-1}b^{-a_2}L(z,a_2+1) \\
                   & \quad \cdots \sum_{a_t=1}^{a_{t-1}-1}b^{-1}[\Phi_b(z)\in H_{a_t-1}]L(z,a_t+1),
  \end{align*}
respectively. By using these forms, the vectors $\bsV(z)$ and $\tilde{\bsV}(z)$ can be computed in $O(\alpha n)$ operations according to \cite[Algorithm~4]{BDLNP12}. Thus the value of $\chi_b\circ \Phi_b(z)$ can be computed in at most $O(\alpha n)$ operations. 
\end{remark}

%%%%%%%%%%%%%%%%%%%%%%%%%%%%%%%%%%%%%%%%%%%%%%%%%%%%%%%%%%%
\subsection{Fast component-by-component construction}
We finally show how to obtain the fast CBC construction using the fast Fourier transform. Our exposition here follows basically along the same lines as \cite[Subsection~6.2]{Gxx2}. By denoting
  \begin{align*}
    P_{\tau-1}(h) := \prod_{j=1}^{\tau-1}\left[ 1+\gamma_j^{1/2}D_{\alpha,b}^{1/2}\chi_b\circ \Phi_b\left(v_n\left(\frac{h(x)q_j(x)}{p(x)}\right)\right)\right] ,
  \end{align*}
and 
  \begin{align*}
    Q(q,h) := \chi_b\circ \Phi_b\left(v_n\left(\frac{h(x)q(x)}{p(x)}\right)\right) ,
  \end{align*}
where the arguments $h,q$ of $P_{\tau-1}$ and $Q$ are understood as integers, and $P_{0}(h)=1$ for any $h\in \nat_0$, we have
  \begin{align*}
    & \quad B_{\alpha}((\bsq_{\tau-1},\tilde{q}_{\tau}),p) \\
    & = -1+\frac{1}{b^m}\sum_{\bsz\in \Pcal((\bsq_{\tau-1},\tilde{q}_{\tau}),p)}\prod_{j=1}^{\tau}\left[ 1+\gamma_j^{1/2}D_{\alpha,b}^{1/2}\chi_b\circ \Phi_b(z_j)\right] \\
    & = -1+\frac{1}{b^m}\sum_{h=0}^{b^m-1}P_{\tau-1}(h)\left[ 1+\gamma_{\tau}^{1/2}D_{\alpha,b}^{1/2}\chi_b\circ \Phi_b\left(v_n\left(\frac{h(x)\tilde{q}_{\tau}(x)}{p(x)}\right)\right)\right] \\
    & = -1+\frac{1}{b^m}\sum_{h=0}^{b^m-1}P_{\tau-1}(h)+\frac{\gamma_{\tau}^{1/2}D_{\alpha,b}^{1/2}}{b^m}\sum_{h=0}^{b^m-1}P_{\tau-1}(h)Q(\tilde{q}_{\tau},h) \\
    & = B_{\alpha}(\bsq_{\tau-1},p)+\frac{\gamma_{\tau}^{1/2}D_{\alpha,b}^{1/2}}{b^m}\left[ P_{\tau-1}(0)Q(\tilde{q}_{\tau},0)+\sum_{h=1}^{b^m-1}P_{\tau-1}(h)Q(\tilde{q}_{\tau},h)\right] ,
  \end{align*}
where the argument $\tilde{q}_{\tau}$ of $Q$ is again understood as an integer. Since we have $Q(q,0) = \chi_b(0)$ for any $q\in R_{b,n}$, we can focus on the term 
  \begin{align}\label{eq:fast_cbc_1}
    \sum_{h=1}^{b^m-1}P_{\tau-1}(h)Q(\tilde{q}_{\tau},h) ,
  \end{align}
and find $\tilde{q}_{\tau}=q_{\tau}$ which minimizes (\ref{eq:fast_cbc_1}) as a function of $\tilde{q}_{\tau}\in R_{b,n}$ in Algorithm \ref{algorithm:cbc}. Moreover, for $\tilde{q}_{\tau}=0$, we have
  \begin{align*}
    B_{\alpha}((\bsq_{\tau-1},0),p) & = -1+\frac{1}{b^m}\sum_{h=0}^{b^m-1}P_{\tau-1}(h)\left[ 1+\gamma_{\tau}^{1/2}D_{\alpha,b}^{1/2}\chi_b(0)\right] \\
    & = -\gamma_{\tau}^{1/2}D_{\alpha,b}^{1/2}\chi_b(0)+\left[ 1+\gamma_{\tau}^{1/2}D_{\alpha,b}^{1/2}\chi_b(0)\right]B_{\alpha}(\bsq_{\tau-1},p) ,
  \end{align*}
which can be computed at a negligibly low cost, so that we only need to consider $\tilde{q}_{\tau}\in R_{b,n}\setminus \{0\}$ in the following.

According to Algorithm \ref{algorithm:cbc}, we choose an irreducible polynomial $p\in \ZZ_b[x]$ with $\deg(p)=n$. Thus, there exists a primitive element $g\in R_{b,n}\setminus \{0\}$, which satisfies
  \begin{align*}
    \{g^0 \bmod{p}, g^1 \bmod{p},\ldots,g^{b^n-2} \bmod{p}\} = R_{b,n}\setminus \{0\} ,
  \end{align*}
and $g^{-1} \bmod{p}=g^{b^n-2} \bmod{p}$. Using this property of $g$, one can obtain the values of (\ref{eq:fast_cbc_1}) for all $\tilde{q}_{\tau}\in R_{b,n}\setminus \{0\}$ by a matrix-vector multiplication $\bsQ_{\mathrm{perm}}\bsP_{\tau-1}$, where $\bsQ_{\mathrm{perm}}$ is a $(b^n-1)\times (b^m-1)$ matrix given by permuting the rows of $\bsQ$ as
  \begin{align*}
    \bsQ_{\mathrm{perm}} :=(Q(g^i \bmod{p},h))_{0\le i\le b^n-2,0<h<b^m} ,
  \end{align*}
and $\bsP_{\tau-1}$ is a vector defined as
  \begin{align*}
    \bsP_{\tau-1} := (P_{\tau-1}(1),\ldots,P_{\tau-1}(b^m-1))^{\top} .
  \end{align*}
However, the multiplication $\bsQ_{\mathrm{perm}}\bsP_{\tau-1}$ requires $O(b^{m+n})$ arithmetic operations, which can be reduced as follows.

As a first step, we add more columns to $\bsQ_{\mathrm{perm}}$ to obtain a $(b^n-1)\times (b^n-1)$ matrix $\bsQ_{\mathrm{perm}}'$ given by
  \begin{align*}
    \bsQ_{\mathrm{perm}}' :=(Q(g^i \bmod{p},h))_{0\le i\le b^n-2,0<h<b^n} .
  \end{align*}
We also add more elements to $\bsP_{\tau-1}$ to obtain a vector $\bsP_{\tau-1}'=(\bsP_{\tau-1}^\top,\bszero^\top)^\top$, where $\bszero$ is the ($b^n-b^m$)-dimensional zero vector, such that we have $\bsQ_{\mathrm{perm}}\bsP_{\tau-1}=\bsQ_{\mathrm{perm}}'\bsP_{\tau-1}'$. Next we permute the rows of $\bsQ_{\mathrm{perm}}'$ to obtain a $(b^n-1)\times (b^n-1)$ \emph{circulant} matrix
  \begin{align*}
    \bsQ_{\mathrm{circ}} & := (Q(g^i \bmod{p},g^{-h} \bmod{p}))_{0\le i,h\le b^n-2} \\
    & = \left( \chi_b\circ \Phi_b\left(v_n\left(\frac{(g^{i-h} \bmod{p})(x)}{p(x)}\right)\right) \right)_{0\le i,h\le b^n-2}  ,
  \end{align*}
where the argument $g^{-h} \bmod{p}$ of $Q$ is again understood as an integer. Further, we introduce a vector $\bsR_{\tau-1}=(R_{\tau-1}(0),\ldots,R_{\tau-1}(b^n-2))^\top$ such that
  \begin{align*}
    R_{\tau-1}(h) = \begin{cases}
    P_{\tau-1}(g^{-h}\bmod{p}) & \text{if $\deg(g^{-h}\bmod{p})<m$},  \\
    0 & \text{otherwise} ,
    \end{cases}
  \end{align*}
for $0\le h\le b^n-2$. Using these notations, we have $\bsQ_{\mathrm{perm}}'\bsP_{\tau-1}'=\bsQ_{\mathrm{circ}}\bsR_{\tau-1}$, which implies that a matrix-vector multiplication $\bsQ_{\mathrm{circ}}\bsR_{\tau-1}$ gives the values of (\ref{eq:fast_cbc_1}) for all $\tilde{q}_{\tau}\in R_{b,n}\setminus \{0\}$. Since the matrix $\bsQ_{\mathrm{circ}}$ is circulant, the multiplication $\bsQ_{\mathrm{circ}}\bsR_{\tau-1}$ can be done efficiently by using the fast Fourier transform, requiring only $O(nb^n)$ arithmetic operations \cite{NC06a}. This way we can reduce the cost from $O(b^{m+n})$ to $O(nb^n)$ operations for finding $q_{\tau}\in R_{b,n}$ which minimizes $B_{\alpha}((\bsq_{\tau-1},\tilde{q}_{\tau}),p)$ as a function of $\tilde{q}_{\tau}$.

Suppose that $g^i\bmod{p}\in R_{b,n}\setminus \{0\}$ minimizes $B_{\alpha}((\bsq_{\tau-1},\tilde{q}_{\tau}),p)$ as a function of $\tilde{q}_{\tau}$. We update $\bsR_{\tau}$ by
  \begin{align*}
    R_{\tau}(h) = \begin{cases}
    & R_{\tau-1}(h) \left[ 1+\gamma_{\tau}^{1/2}D_{\alpha,b}^{1/2} Q(g^i \bmod{p},g^{-h} \bmod{p})\right] \\
    & \qquad \text{if $\deg(g^{-h}\bmod{p})<m$} ,  \\
    & \\ 
    & 0 \quad \text{otherwise} ,
    \end{cases}
  \end{align*}
for $0\le h\le b^n-2$. We then move on to the next component. In summary, the fast CBC construction proceeds as follows.

    \begin{enumerate}
        \item Choose an irreducible polynomial $p\in \ZZ_b[x]$ with $\deg(p)=n$.
        \item Evaluate $\bsQ_{\mathrm{circ}}$ and $\bsR_{0}$.
        \item For $\tau=1,\ldots, s$, do the following:
    \begin{itemize}
        \item Compute the matrix-vector multiplication $\bsQ_{\mathrm{circ}}\bsR_{\tau-1}$ by using the fast Fourier transform.
        \item Find $q_{\tau}\in R_{b,n}\setminus \{0\}$ which gives the minimum value among the components of $\bsQ_{\mathrm{circ}}\bsR_{\tau-1}$.
        \item Update the vector $\bsR_{\tau}$.
    \end{itemize}
    \end{enumerate}

In Step 2, since the matrix $\bsQ_{\mathrm{circ}}$ is circulant, we only need to evaluate one column of $\bsQ_{\mathrm{circ}}$. Since one column consists of $b^n-1$ elements, each of which can be computed in at most $O(\alpha n)$ operations as in Theorem \ref{theorem:chi_calculation}, Step 2 can be done in at most $O(\alpha nb^n)$ operations. The matrix-vector multiplication $\bsQ_{\mathrm{circ}}\bsR_{\tau-1}$ can be done in $O(nb^n)$ operations and $\bsR_{\tau}$ can be updated in $O(b^n)$ operations, so that Step 3 can be done in $O(snb^n)$ operations. In total, the fast CBC construction can be done in $O(snb^n)$ operations. As for the required memory, we need to store one column of $\bsQ_{\mathrm{circ}}$ and $\bsR_{\tau}$, both of which require $O(b^n)$ memory space.

We now recall that we can construct good FHOPLPSs which achieve the optimal rate of the worst-case error when $n\ge \alpha m/2$, see Remark \ref{remark:cbc}. Depending on whether $\alpha m$ is even or odd, we may choose $n=\alpha m/2$ or $n=(\alpha m+1)/2$, respectively. In both cases, the computational cost of the fast CBC construction becomes $O(s\alpha mb^{\alpha m/2})=O(s\alpha N^{\alpha /2}\log N)$ operations using $O(b^{\alpha m/2})=O(N^{\alpha/2})$ memory.

\begin{remark}\label{remark:cbc_digital_shift}
In \cite{GSYxx}, the authors studied the mean square worst-case error of \emph{digitally shifted and then folded} higher order polynomial lattice point sets in the same function space $\Hcal_{\alpha}$ and proved the existence of good higher order polynomial lattice rules which achieve the optimal convergence rate. Following an argument similar to the previous section, it is possible to prove that the CBC construction can find good higher order polynomial lattice rules which achieve the optimal convergence rate. Moreover, the fast CBC construction can also be obtained by slightly modifying the argument of this section.
\end{remark}

%%%%%%%%%%%%%%%%%%%%%%%%%%%%%%%%%%%%%%%%%%%%%%%%%%%%%%%%%%%
%%%%%%%%%%%%%%%%%%%%%%%%%%%%%%%%%%%%%%%%%%%%%%%%%%%%%%%%%%%
%%%%%%%%%%%%%%%%%%%%%%%%%%%%%%%%%%%%%%%%%%%%%%%%%%%%%%%%%%%
\appendix

%%%%%%%%%%%%%%%%%%%%%%%%%%%%%%%%%%%%%%%%%%%%%%%%%%%%%%%%%%%
%%%%%%%%%%%%%%%%%%%%%%%%%%%%%%%%%%%%%%%%%%%%%%%%%%%%%%%%%%%
%%%%%%%%%%%%%%%%%%%%%%%%%%%%%%%%%%%%%%%%%%%%%%%%%%%%%%%%%%%
\section{Proof of Theorem~\ref{theorem:cbc}}\label{appendix:a}
In order to prove Theorem~\ref{theorem:cbc}, we need the following lemma, see \cite[Lemma~25]{GSYxx} for the proof.

\begin{lemma}\label{lem:sum-of-digit}
Let $b$ be a prime, let $\alpha\ge 2$ be an integer. For $n\in \nat$ and a real number $\lambda >1/\alpha$, we have
\begin{align*}
    \sum_{k\in \Ecal}b^{-\lambda\mu_{\alpha}(\lfloor k/b\rfloor)}\le A_{\alpha,b,\lambda,1} \quad \mathrm{and}\quad \sum_{\substack{k\in \Ecal\\ b^n\mid k}}b^{-\lambda\mu_{\alpha}(\lfloor k/b\rfloor)}\le \frac{A_{\alpha,b,\lambda,2}}{b^{2\lambda n}},
  \end{align*}
where $A_{\alpha,b,\lambda,1}$ and $A_{\alpha,b,\lambda,2}$ are positive and depend only on $\alpha$, $b$ and $\lambda$, which are respectively given by
  \begin{align*}
    A_{\alpha,b,\lambda,1}=\frac{b}{b-1}\left[ \sum_{v=1}^{\alpha-1}\prod_{i=1}^{v}\left( \frac{b-1}{b^{\lambda i}-1}\right)+\frac{b^{\lambda \alpha}-1}{b^{\lambda \alpha}-b}\prod_{i=1}^{\alpha}\left( \frac{b-1}{b^{\lambda i}-1}\right)\right] ,
  \end{align*}
and
  \begin{align*}
    A_{\alpha,b,\lambda,2}=\frac{1}{b-1}\sum_{v=2}^{\alpha-1}\prod_{i=1}^{v}\left( \frac{b^{\lambda}(b-1)}{b^{\lambda i}-1}\right)+\frac{b^{\lambda}}{b^{\lambda\alpha}-b}\prod_{i=1}^{\alpha-1}\left( \frac{b^{\lambda}(b-1)}{b^{\lambda i}-1}\right) .
  \end{align*}
\end{lemma}

Here we note that $A_{\alpha,b,\lambda}$ in Theorem \ref{theorem:cbc} is given by $A_{\alpha,b,\lambda}=A_{\alpha,b,\lambda,1}+A_{\alpha,b,\lambda,2}$. In the following argument, we shall use the inequality
  \begin{align}\label{eq:jensen}
    \left( \sum_{n}a_n\right)^{\lambda} \le \sum_{n}a_n^{\lambda} ,
  \end{align}
for any sequence of non-negative real numbers $(a_n)_{n\in \nat}$ and any $0<\lambda\le 1$.

We now prove Theorem~\ref{theorem:cbc} by induction. Let us consider the case $\tau=1$ first. There exists at least one polynomial $q_1\in R_{b,n}$ for which $B_{\alpha}^{\lambda}(q_1,p)$ is smaller than or equal to the average of $B_{\alpha}^{\lambda}(\tilde{q}_1,p)$ over $\tilde{q}_1\in R_{b,n}$. Thus, we have
  \begin{align}\label{eq:cbc_proof1}
    B_{\alpha}^{\lambda}(q_1,p) & \le \frac{1}{b^n}\sum_{\tilde{q}_1\in R_{b,n}}B_{\alpha}^{\lambda}(\tilde{q}_1,p) \nonumber \\
    & \le \gamma_1^{\lambda/2}D_{\alpha,b}^{\lambda/2}\sum_{k_1\in \Ecal}b^{-\lambda \mu_{\alpha}(\lfloor k_1/b\rfloor)}\frac{1}{b^n}\sum_{\substack{\tilde{q}_1\in R_{b,n} \\ \tr_n(k_1)\cdot \tilde{q}_1 \equiv a \pmod p\\ \deg(a)<n-m}}1 ,
  \end{align}
for $0< \lambda \le 1$. The innermost sum equals the number of solutions $\tilde{q}_1\in R_{b,n}$ such that $\tr_{n}(k_1)\cdot \tilde{q}_1\equiv a \pmod p$ with $\deg(a)<n-m$. If $\tr_n(k_1)$ is a multiple of $p$, we have $\tr_{n}(k_1)\cdot \tilde{q}_1\equiv 0 \pmod p$ independently of $\tilde{q}_1$, so that we have
  \begin{align*}
    \frac{1}{b^n}\sum_{\substack{\tilde{q}_1\in R_{b,n} \\ \tr_n(k_1)\cdot \tilde{q}_1 \equiv a \pmod p\\ \deg(a)<n-m}}1 = 1 .
  \end{align*}
Otherwise if $\tr_n(k_1)$ is not a multiple of $p$, then there are $b^{n-m}$ possible choices for $a\in \ZZ_b[x]$ such that $\deg(a)<n-m$, for each of which there is one solution $\tilde{q}_1$ to $\tr_{n}(k_1)\cdot \tilde{q}_1\equiv a \pmod p$, so that we have
  \begin{align*}
    \frac{1}{b^n}\sum_{\substack{\tilde{q}_1\in R_{b,n} \\ \tr_n(k_1)\cdot \tilde{q}_1 \equiv a \pmod p\\ \deg(a)<n-m}}1 = \frac{1}{b^m} .
  \end{align*}
Substituting these results into (\ref{eq:cbc_proof1}) and using Lemma \ref{lem:sum-of-digit}, we obtain
  \begin{align*}
    B_{\alpha}^{\lambda}(q_1,p) & \le \gamma_1^{\lambda/2}D_{\alpha,b}^{\lambda/2}\left( \sum_{\substack{k_1\in \Ecal\\ b^n\mid k_1}}b^{-\lambda \mu_{\alpha}(\lfloor k_1/b\rfloor)} + \frac{1}{b^m}\sum_{\substack{k_1\in \Ecal\\ b^n\nmid k_1}}b^{-\lambda \mu_{\alpha}(\lfloor k_1/b\rfloor)}\right) \\
    & \le \gamma_1^{\lambda/2}D_{\alpha,b}^{\lambda/2}\left( \frac{A_{\alpha,b,\lambda,2}}{b^{2\lambda n}} + \frac{A_{\alpha,b,\lambda,1}}{b^m}\right) \\
    & \le \frac{\gamma_1^{\lambda/2}D_{\alpha,b}^{\lambda/2}}{b^{\min(m,2\lambda n)}}(A_{\alpha,b,\lambda,1}+A_{\alpha,b,\lambda,2}) = \frac{\gamma_1^{\lambda/2}D_{\alpha,b}^{\lambda/2}A_{\alpha,b,\lambda}}{b^{\min(m,2\lambda n)}},
  \end{align*}
for $1/\alpha<\lambda \le 1$. Hence the result for the case $\tau=1$ follows.

Next we suppose that for $1\le \tau<s$, the inequality
  \begin{align}\label{eq:cbc_proof2}
    B_{\alpha}(\bsq_{\tau},p) \le \frac{1}{b^{\min( m/\lambda, 2n)}}\left[-1+\prod_{j=1}^{\tau}\left(1+\gamma_j^{\lambda/2}D_{\alpha,b}^{\lambda/2}A_{\alpha,b,\lambda}\right)\right]^{1/\lambda}
  \end{align}
holds true for any $1/\alpha<\lambda \le 1$. Then we have
  \begin{align}\label{eq:cbc_proof3}
    & \quad B_{\alpha}((\bsq_{\tau},\tilde{q}_{\tau+1}),p) \nonumber \\
    & = \sum_{\emptyset \ne u\subseteq \{1,\ldots,\tau+1\}}\gamma_u^{1/2}D_{\alpha,b}^{|u|/2}\sum_{\substack{\bsk_u\in \Ecal^{|u|}\\ (\bsk_u,\bszero) \in \Pcal^{\perp}((\bsq_{\tau},\tilde{q}_{\tau+1}),p)}}b^{-\mu_{\alpha}(\lfloor\bsk_u/b\rfloor)} \nonumber \\
    & = \sum_{\emptyset \ne u\subseteq \{1,\ldots,\tau\}}\gamma_u^{1/2}D_{\alpha,b}^{|u|/2}\sum_{\substack{\bsk_u\in \Ecal^{|u|}\\ (\bsk_u,\bszero) \in \Pcal^{\perp}((\bsq_{\tau},\tilde{q}_{\tau+1}),p)}}b^{-\mu_{\alpha}(\lfloor\bsk_u/b\rfloor)} \nonumber \\
    & \quad + \sum_{u\subseteq \{1,\ldots,\tau\}}\gamma_{u\cup \{\tau+1\}}^{1/2}D_{\alpha,b}^{(|u|+1)/2}\sum_{\substack{\bsk_{u\cup \{\tau+1\}}\in \Ecal^{|u|+1}\\ (\bsk_{u\cup \{\tau+1\}},\bszero) \in\Pcal^{\perp}((\bsq_{\tau},\tilde{q}_{\tau+1}),p)}}b^{-\mu_{\alpha}(\lfloor\bsk_{u\cup \{\tau+1\}}/b\rfloor)} \nonumber \\
    & =: B_{\alpha}(\bsq_{\tau},p)+\theta(\bsq_{\tau},\tilde{q}_{\tau+1},p) ,
  \end{align}
where we denote by $\theta(\bsq_{\tau},\tilde{q}_{\tau+1},p)$ the second term in the last equality. In Algorithm \ref{algorithm:cbc}, we choose $q_{\tau+1}\in R_{b,n}$ which minimizes $\theta(\bsq_{\tau},\tilde{q}_{\tau+1},p)$ as a function of $\tilde{q}_{\tau+1}$, since the dependence of $B_{\alpha}((\bsq_{\tau},\tilde{q}_{\tau+1}),p)$ on $\tilde{q}_{\tau+1}$ appears only in $\theta(\bsq_{\tau},\tilde{q}_{\tau+1},p)$. Using an averaging argument and the inequality (\ref{eq:jensen}), we have
  \begin{align*}
    \theta^{\lambda}(\bsq_{\tau},q_{\tau+1},p) & \le \frac{1}{b^n}\sum_{\tilde{q}_{\tau+1}\in R_{b,n}}\theta^{\lambda}(\bsq_{\tau},\tilde{q}_{\tau+1},p) \\
    & \le \sum_{u\subseteq \{1,\ldots,\tau\}}\gamma_{u\cup \{\tau+1\}}^{\lambda/2}D_{\alpha,b}^{\lambda(|u|+1)/2}\sum_{\bsk_{u\cup \{\tau+1\}}\in \Ecal^{|u|+1}}b^{-\lambda \mu_{\alpha}(\lfloor\bsk_{u\cup \{\tau+1\}}/b\rfloor)} \\
    & \quad \times \frac{1}{b^n}\sum_{\substack{\tilde{q}_{\tau+1}\in R_{b,n}\\ \tr_n(\bsk_u)\cdot \bsq_u + \tr_n(k_{\tau+1})\cdot \tilde{q}_{\tau+1}\equiv a \pmod p\\ \deg(a)<n-m}}1 ,
  \end{align*}
for $0<\lambda \le 1$. If $\tr_n(k_{\tau+1})$ is a multiple of $p$, then we have
  \begin{align*}
    \tr_n(\bsk_u)\cdot \bsq_u + \tr_n(k_{\tau+1})\cdot \tilde{q}_{\tau+1} \equiv \tr_n(\bsk_u)\cdot \bsq_u \pmod p ,
  \end{align*}
so that the innermost sum equals $b^{n}$ for $(\bsk_u,\bszero) \in \Pcal^{\perp}(\bsq_{\tau},p)$, and equals 0 otherwise. If $\tr_{n}(k_{\tau+1})$ is not a multiple of $p$, there are $b^{n-m}$ possible choices for $a$ such that $\deg(a)<n-m$, for each of which there exists at most one solution $\tilde{q}_{\tau+1}$ to $\tr_n(k_{\tau+1})\cdot \tilde{q}_{\tau+1}\equiv a-\tr_n(\bsk_u)\cdot \bsq_u \pmod p$, so that the innermost sum is bounded above by $b^{n-m}$. From these results and using Lemma \ref{lem:sum-of-digit}, we have
  \begin{align}\label{eq:cbc_proof4}
    & \quad \theta^{\lambda}(\bsq_{\tau},q_{\tau+1},p) \nonumber \\
    & \le \sum_{u\subseteq \{1,\ldots,\tau\}}\gamma_{u\cup \{\tau+1\}}^{\lambda/2}D_{\alpha,b}^{\lambda(|u|+1)/2}\sum_{\substack{\bsk_{u\cup \{\tau+1\}}\in \Ecal^{|u|+1}\\ b^n\mid k_{\tau+1}}}b^{-\lambda \mu_{\alpha}(\lfloor\bsk_{u\cup \{\tau+1\}}/b\rfloor)} \nonumber \\
    & \quad + \frac{1}{b^m}\sum_{u\subseteq \{1,\ldots,\tau\}}\gamma_{u\cup \{\tau+1\}}^{\lambda/2}D_{\alpha,b}^{\lambda(|u|+1)/2}\sum_{\substack{\bsk_{u\cup \{\tau+1\}}\in \Ecal^{|u|+1}\\ b^n\nmid k_{\tau+1}}}b^{-\lambda \mu_{\alpha}(\lfloor\bsk_{u\cup \{\tau+1\}}/b\rfloor)} \nonumber \\
    & \le \sum_{u\subseteq \{1,\ldots,\tau\}}\gamma_{u\cup \{\tau+1\}}^{\lambda/2}D_{\alpha,b}^{\lambda(|u|+1)/2}\sum_{\bsk_u\in \Ecal^{|u|}}b^{-\lambda \mu_{\alpha}(\lfloor\bsk_u/b\rfloor)} \nonumber \\
    & \quad \times \left( \sum_{\substack{k_{\tau+1}\in \Ecal\\ b^n\mid k_{\tau+1}}}b^{-\lambda \mu_{\alpha}(\lfloor k_{\tau+1}/b\rfloor )} + \frac{1}{b^m}\sum_{k_{\tau+1}\in \Ecal}b^{-\lambda \mu_{\alpha}(\lfloor k_{\tau+1}/b\rfloor )}\right) \nonumber \\
    & \le \sum_{u\subseteq \{1,\ldots,\tau\}}\gamma_{u\cup \{\tau+1\}}^{\lambda/2}D_{\alpha,b}^{\lambda(|u|+1)/2}A_{\alpha,b,\lambda,1}^{|u|}\left( \frac{A_{\alpha,b,\lambda,2}}{b^{2\lambda n}} + \frac{A_{\alpha,b,\lambda,1}}{b^m}\right) \nonumber \\
    & \le \frac{\gamma_{\tau+1}^{\lambda/2}D_{\alpha,b}^{\lambda/2}A_{\alpha,b,\lambda}}{b^{\min(m,2\lambda n)}}\sum_{u\subseteq \{1,\ldots,\tau\}}\gamma_u^{\lambda/2}D_{\alpha,b}^{\lambda |u|/2}A_{\alpha,b,\lambda}^{|u|} \nonumber \\
    & = \frac{\gamma_{\tau+1}^{\lambda/2}D_{\alpha,b}^{\lambda/2}A_{\alpha,b,\lambda}}{b^{\min(m,2\lambda n)}}\prod_{j=1}^{\tau}\left(1+\gamma_j^{\lambda/2}D_{\alpha,b}^{\lambda/2}A_{\alpha,b,\lambda}\right) ,
  \end{align}
for $1/\alpha <\lambda \le 1$. Applying the inequality (\ref{eq:jensen}) to (\ref{eq:cbc_proof3}) in which $q_{\tau+1}$ equals that $\tilde{q}_{\tau+1}$ which minimizes $\theta^{\lambda}(\bsq_{\tau},\tilde{q}_{\tau+1},p)$, and then using (\ref{eq:cbc_proof2}) and (\ref{eq:cbc_proof4}), we obtain
  \begin{align*}
    B_{\alpha}^{\lambda}(\bsq_{\tau+1},p) & \le B_{\alpha}^{\lambda}(\bsq_{\tau},p)+\theta^{\lambda}(\bsq_{\tau},q_{\tau+1},p) \\
    & \le \frac{1}{b^{\min( m, 2\lambda n)}}\left[-1+\prod_{j=1}^{\tau}\left(1+\gamma_j^{\lambda/2}D_{\alpha,b}^{\lambda/2}A_{\alpha,b,\lambda}\right)\right] \\
    & \quad + \frac{\gamma_{\tau+1}^{\lambda/2}D_{\alpha,b}^{\lambda/2}A_{\alpha,b,\lambda}}{b^{\min(m,2\lambda n)}}\prod_{j=1}^{\tau}\left(1+\gamma_j^{\lambda/2}D_{\alpha,b}^{\lambda/2}A_{\alpha,b,\lambda}\right) \\
    & = \frac{1}{b^{\min( m, 2\lambda n)}}\left[-1+\prod_{j=1}^{\tau+1}\left(1+\gamma_j^{\lambda/2}D_{\alpha,b}^{\lambda/2}A_{\alpha,b,\lambda}\right)\right] ,
  \end{align*}
for $1/\alpha<\lambda \le 1$, which completes the proof.

%%%%%%%%%%%%%%%%%%%%%%%%%%%%%%%%%%%%%%%%%%%%%%%%%%%%%%%%%%%
%%%%%%%%%%%%%%%%%%%%%%%%%%%%%%%%%%%%%%%%%%%%%%%%%%%%%%%%%%%
%%%%%%%%%%%%%%%%%%%%%%%%%%%%%%%%%%%%%%%%%%%%%%%%%%%%%%%%%%%
\section{Proof of Theorem~\ref{theorem:chi_calculation}}\label{appendix:b}
Since $z\in \Zbinfty$ is given in the form $(\zeta_1,\zeta_2,\ldots,\zeta_n,0,0,\ldots)^{\top}$ for $n \in \nat$ and $\zeta_i\in \ZZ_b$, $1\le i\le n$, $\Phi_b(z)$ is given as
  \begin{align*}
     \Phi_b(z)=(\eta_1,\eta_2,\ldots)^{\top}\in \Zbinfty\quad \mathrm{with}\quad \eta_i = \zeta_{i+1}-\zeta_{1} \pmod b ,
  \end{align*}
where $\zeta_{n+1}=\zeta_{n+2}=\cdots=0$. 

Let us consider the first part. If $\zeta_1=\cdots=\zeta_n=0$, it holds that $\Phi_b(z)=(0,0,\ldots)^{\top}$, and thus, $W_k(\Phi_b(z))=1$ for all $k\in \nat$. Therefore, we have
  \begin{align*}
     \chi_b\circ \Phi_b(z) = \sum_{k=1}^{\infty}b^{-\mu_{\alpha}(k)},
  \end{align*}
where the result of the last sum is given in \cite[Theorem~2]{BDLNP12}, which proves the first part.

Let us consider the second part next. We denote the $b$-adic expansion of $k\in \nat$ by $k=\kappa_1b^{a_1-1}+\cdots+\kappa_v b^{a_v-1}$ for some $v\in \nat$ such that $a_1>\cdots >a_v>0$ and $0<\kappa_1,\ldots,\kappa_v<b$. Then we have $\mu_{\alpha}(k)=a_1+\dots+a_{\min(v,\alpha)}$ and 
  \begin{align*}
     W_k(\Phi_b(z))=\prod_{i=1}^{v}\omega_b^{\kappa_i(\zeta_{a_i+1}-\zeta_{1})} .
  \end{align*}
Here we note that $\mu_{\alpha}(k)$ does not depend on the values of $\kappa_1,\ldots,\kappa_v$. Using this result and arranging every element of $\nat$ according to the value of $v$ in their expansions, we obtain 
  \begin{align}\label{eq:eff_calcul_1}
    & \quad \chi_b\circ \Phi_b(z) \nonumber \\
    & = \sum_{v=1}^{\infty}\sum_{0<a_v<\cdots <a_1}b^{-\sum_{i=1}^{\min(v,\alpha)}a_i}\sum_{0<\kappa_1,\ldots,\kappa_v<b}\prod_{i=1}^{v}\omega_b^{\kappa_i(\zeta_{a_i+1}-\zeta_{1})} \nonumber \\
    & = \sum_{v=1}^{\alpha-1}\sum_{0<a_v<\cdots <a_1}\prod_{i=1}^{v}\left(b^{-a_i}\sum_{\kappa_i=1}^{b-1}\omega_b^{\kappa_i(\zeta_{a_i+1}-\zeta_{1})}\right) \nonumber \\
    & \quad + \sum_{v=\alpha}^{\infty}\sum_{0<a_v<\cdots <a_1}\prod_{i=1}^{\alpha}\left(b^{-a_i}\sum_{\kappa_i=1}^{b-1}\omega_b^{\kappa_i(\zeta_{a_i+1}-\zeta_{1})}\right) \prod_{j=\alpha+1}^{v}\left(\sum_{\kappa_j=1}^{b-1}\omega_b^{\kappa_j(\zeta_{a_j+1}-\zeta_{1})}\right) ,
  \end{align}
wherein we have
  \begin{align*}
    \sum_{\kappa_i=1}^{b-1}\omega_b^{\kappa_i(\zeta_{a_i+1}-\zeta_{1})} & = \begin{cases}
     b-1 & (\zeta_{a_i+1}-\zeta_{1}=0) \\
     -1  & (\zeta_{a_i+1}-\zeta_{1}\ne 0)
    \end{cases} \\
    & = L(z,a_i+1) .
  \end{align*}
Thus, the second term on the right-hand side of (\ref{eq:eff_calcul_1}) becomes
  \begin{align}\label{eq:eff_calcul_2}
    & \quad \sum_{v=\alpha}^{\infty}\sum_{0<a_v<\cdots <a_1}\prod_{i=1}^{\alpha}b^{-a_i}L(z,a_i+1) \prod_{j=\alpha+1}^{v}L(z,a_j+1) \nonumber \\
    & = \sum_{0<a_\alpha <\cdots <a_1}\prod_{i=1}^{\alpha}b^{-a_i}L(z,a_i+1) \nonumber \\
    & \quad \times  \sum_{v=\alpha}^{\infty}\sum_{0< a_v< \cdots <a_{\alpha+1}<a_{\alpha}}\prod_{i'\in \{a_{v},\ldots,a_{\alpha+1}\}}L(z,i'+1)\prod_{i''\in \{1,\ldots,a_{\alpha}-1\}\setminus \{a_{v},\ldots,a_{\alpha+1}\}}1  \nonumber \\
    & = \sum_{0<a_\alpha <\cdots <a_1}\prod_{i=1}^{\alpha}b^{-a_i}L(z,a_i+1)\sum_{u\subseteq \{1,\ldots,a_{\alpha}-1\}}\prod_{i'\in u}L(z,i'+1)\prod_{i''\in \{1,\ldots,a_{\alpha}-1\}\setminus u}1 \nonumber \\
    & = \sum_{0<a_\alpha <\cdots <a_1}\prod_{i=1}^{\alpha}b^{-a_i}L(z,a_i+1)\prod_{j=1}^{a_{\alpha}-1}\left[ 1+L(z,j+1)\right] .
  \end{align}
The innermost product equals $b^{a_{\alpha}-1}$ if and only if $L(z,j+1)=b-1$ for all $1\le j< a_{\alpha}$ and equals 0 otherwise. Since $L(z,j+1)=b-1$ only when $\zeta_{j+1}=\zeta_{1}$, we focus on the condition $\zeta_{j+1}=\zeta_{1}$ for all $1\le j< a_{\alpha}$. It is obvious that $z\in \Zbinfty$ satisfying this condition can be expressed in the form
  \begin{align*}
    (\underbrace{\zeta_1, \zeta_1,\ldots ,\zeta_1}_{a_{\alpha}}, \zeta_{a_{\alpha}+1},\zeta_{a_{\alpha}+2},\ldots)^{\top} .
  \end{align*}
For such $z$ we have 
  \begin{align*}
    \Phi_b(z) = (\underbrace{0, 0,\ldots ,0}_{a_{\alpha}-1}, \eta_{a_{\alpha}+1},\eta_{a_{\alpha}+2},\ldots)^{\top} .
  \end{align*}
Therefore, the innermost product on the right-most side of (\ref{eq:eff_calcul_2}) equals $b^{a_{\alpha}-1}$ if $\Phi_b(z)\in H_{a_{\alpha}-1}$, and equals 0 otherwise. Hence the second term on the right-hand side of (\ref{eq:eff_calcul_1}) can be further rewritten as
  \begin{align*}
    \sum_{0<a_\alpha <\cdots <a_1}b^{a_{\alpha}-1}[\Phi_b(z)\in H_{a_{\alpha}-1}]\prod_{i=1}^{\alpha}b^{-a_i}L(z,a_i+1)  .
  \end{align*}
Substituting this result into (\ref{eq:eff_calcul_1}) we have
  \begin{align}\label{eq:eff_calcul_3}
    \chi_b\circ \Phi_b(z) & = \sum_{v=1}^{\alpha-1}\sum_{0<a_v<\cdots <a_1}\prod_{i=1}^{v}b^{-a_i}L(z,a_i+1) \nonumber \\
    & \quad + \sum_{0<a_\alpha <\cdots <a_1}b^{a_{\alpha}-1}[\Phi_b(z)\in H_{a_{\alpha}-1}]\prod_{i=1}^{\alpha}b^{-a_i}L(z,a_i+1) .
  \end{align}

For the first term on the right-hand side of (\ref{eq:eff_calcul_3}) we have
  \begin{align*}
    & \quad \sum_{0<a_v<\cdots <a_1}\prod_{i=1}^{v}b^{-a_i}L(z,a_i+1) \\
    & = \sum_{0<a_v<\cdots <a_1<n}\prod_{i=1}^{v}b^{-a_i}L(z,a_i+1)+ \sum_{0<a_v<\cdots <a_2<n\le a_1}\prod_{i=1}^{v}b^{-a_i}L(z,a_i+1) \\
    & \quad + \cdots + \sum_{0<a_v<n\le a_{v-1}< \cdots <a_1}\prod_{i=1}^{v}b^{-a_i}L(z,a_i+1)+ \sum_{n\le a_v<\cdots <a_1}\prod_{i=1}^{v}b^{-a_i}L(z,a_i+1) \\
    & = \sum_{t=0}^{v}\sum_{0<a_v<\cdots <a_{t+1}<n}\prod_{i=t+1}^{v}b^{-a_i}L(z,a_i+1)\sum_{n\le a_t<\cdots <a_1}\prod_{j=1}^{t}b^{-a_j}L(z,a_j+1) , \\
    & = \sum_{t=0}^{v}V_{v-t}(z)\sum_{n\le a_t<\cdots <a_1}\prod_{j=1}^{t}b^{-a_j}L(z,a_j+1) ,
  \end{align*}
where we define $V_{0}(z)=1$ for any $z\in G$. We now recall that $z\in \Zbinfty$ is given in the form $(\zeta_1,\zeta_2,\ldots,\zeta_n,0,0,\ldots)^{\top}$ for $n \in \nat$ and $\zeta_i\in \ZZ_b$, $1\le i\le n$. From this assumption, we have $\zeta_{n+1}=\zeta_{n+2}=\cdots =0$, so that for any $i\ge n$
  \begin{align*}
    L(z,i+1) & = \begin{cases}
     b-1 & (\zeta_{1}=0)  \\
     -1 & (\zeta_{1}\ne 0)
    \end{cases} \\
    & = \rho(\zeta_{1}) .
  \end{align*}
Therefore in the last expression we have
  \begin{align}\label{eq:U_v}
    \sum_{n\le a_t<\cdots <a_1}\prod_{j=1}^{t}b^{-a_j}L(z,a_j+1) & = \rho^t(\zeta_{1})\sum_{n\le a_t<\cdots <a_1}\prod_{j=1}^{t}b^{-a_j} \nonumber \\
    & = \rho^t(\zeta_{1})\sum_{a_t=n}^{\infty}b^{-a_{t}}\sum_{a_{t-1}=a_t+1}^{\infty}b^{-a_{t-1}}\cdots \sum_{a_1=a_2+1}^{\infty}b^{-a_{1}} \nonumber \\
    & = \frac{1}{b^{t(n-1)}}\prod_{i=1}^{t}\frac{\rho(\zeta_{1})}{b^i-1} = U_t(\zeta_1) .
  \end{align}
Using these results and swapping the order of sums, the first term on the right-hand side of (\ref{eq:eff_calcul_3}) becomes
  \begin{align*}
    \sum_{v=1}^{\alpha-1}\sum_{0<a_v<\cdots <a_1}\prod_{i=1}^{v}b^{-a_i}L(z,a_i+1) & = \sum_{v=1}^{\alpha-1}\sum_{t=0}^{v}V_{v-t}(z)U_t(\zeta_1) \\
    & = \sum_{t=1}^{\alpha-1}\left(\sum_{v=t}^{\alpha-1}U_{v-t}(\zeta_1)\right)V_t(z)+\sum_{v=1}^{\alpha-1}U_{v}(\zeta_1) \\
    & = \sum_{t=1}^{\alpha-1}\tilde{U}_t(\zeta_1)V_t(z)+(\tilde{U}_0(\zeta_1)-1) .
  \end{align*}
  
For the second term on the right-hand side of (\ref{eq:eff_calcul_3}), we have
  \begin{align*}
    & \quad \sum_{0<a_\alpha <\cdots <a_1}b^{a_{\alpha}-1}[\Phi_b(z)\in H_{a_{\alpha}-1}]\prod_{i=1}^{\alpha}b^{-a_i}L(z,a_i+1) \\
    & = \sum_{0<a_\alpha <\cdots <a_1<n}b^{a_{\alpha}-1}[\Phi_b(z)\in H_{a_{\alpha}-1}]\prod_{i=1}^{\alpha}b^{-a_i}L(z,a_i+1) \\
    & \quad + \sum_{0<a_\alpha <\cdots <a_2<n\le a_1}b^{a_{\alpha}-1}[\Phi_b(z)\in H_{a_{\alpha}-1}]\prod_{i=1}^{\alpha}b^{-a_i}L(z,a_i+1) \\
    & \quad +\cdots + \sum_{n\le a_\alpha <\cdots <a_1}b^{a_{\alpha}-1}[\Phi_b(z)\in H_{a_{\alpha}-1}]\prod_{i=1}^{\alpha}b^{-a_i}L(z,a_i+1) \\
    & = \sum_{t=0}^{\alpha}\sum_{0<a_\alpha<\cdots <a_{t+1}<n}b^{a_{\alpha}-1}[\Phi_b(z)\in H_{a_{\alpha}-1}]\prod_{i=t+1}^{\alpha}b^{-a_i}L(z,a_i+1) \\
    & \quad \times \sum_{n\le a_t<\cdots <a_1}\prod_{j=1}^{t}b^{-a_i}L(z,a_i+1) .
  \end{align*}
Here again we recall that $z\in \Zbinfty$ is given in the form $(\zeta_1,\zeta_2,\ldots,\zeta_n,0,0,\ldots)^{\top}$ for $n \in \nat$ and $\zeta_i\in \ZZ_b$, $1\le i\le n$. Furthermore, it does not hold that $\zeta_1=\cdots=\zeta_n=0$ for the second part of this theorem. Therefore, $\Phi_b(z)\notin H_{a_{\alpha}-1}$ whenever $a_{\alpha}\ge n$. Thus, we have
  \begin{align*}
    \sum_{n\le a_\alpha <\cdots <a_1}b^{a_{\alpha}-1}[\Phi_b(z)\in H_{a_{\alpha}-1}]\prod_{i=1}^{\alpha}b^{-a_i}L(z,a_i+1) = 0.
  \end{align*}
Thus, by using the above result and (\ref{eq:U_v}), the second term on the right-hand side of (\ref{eq:eff_calcul_3}) becomes
  \begin{align*}
    & \sum_{0<a_\alpha <\cdots <a_1}b^{a_{\alpha}-1}[\Phi_b(z)\in H_{a_{\alpha}-1}]\prod_{i=1}^{\alpha}b^{-a_i}L(z,a_i+1) \\
    = & \sum_{t=0}^{\alpha-1}\tilde{V}_{\alpha-t}(z)U_t(\zeta_1) = \sum_{t=1}^{\alpha}U_{\alpha-t}(\zeta_1)\tilde{V}_{t}(z) .
  \end{align*}
Therefore, we have
  \begin{align*}
    \chi_b\circ \Phi_b(z) = \sum_{t=1}^{\alpha-1}\tilde{U}_t(\zeta_1)V_t(z)+(\tilde{U}_0(\zeta_1)-1) + \sum_{t=1}^{\alpha}U_{\alpha-t}(\zeta_1)\tilde{V}_{t}(z) ,
  \end{align*}
which completes the proof of the second part.

%%%%%%%%%%%%%%%%%%%%%%%%%%%%%%%%%%%%%%%%%%%%%%%%%%%%%%%%%%%
%%%%%%%%%%%%%%%%%%%%%%%%%%%%%%%%%%%%%%%%%%%%%%%%%%%%%%%%%%%
%%%%%%%%%%%%%%%%%%%%%%%%%%%%%%%%%%%%%%%%%%%%%%%%%%%%%%%%%%%

\end{document}